\documentclass[hidelinks,onefignum,onetabnum]{siamart250211}

\usepackage[utf8]{inputenc} 
\usepackage[T1]{fontenc}    
\usepackage{nicefrac}       
\usepackage{amsmath,amssymb,amsfonts}        
\usepackage{pifont}         
\usepackage{enumitem}       
\usepackage{graphicx}       
\usepackage{bbm}
\usepackage{algorithm}      
\usepackage{algpseudocode}  

\ifpdf
  \DeclareGraphicsExtensions{.eps,.pdf,.png,.jpg}
\else
  \DeclareGraphicsExtensions{.eps}
\fi

\def\R{{\mathbb{R}}}

\def\N{{\mathbb{N}}}

\definecolor{darkgreen}{rgb}{0.1, .6, 0.}

\newcommand{\Prox}[2]{\mathsf{Prox}_{#1}\left(#2\right)}

\DeclareMathOperator*{\argmin}{arg\,min}

\newsiamremark{remark}{Remark}
\newsiamremark{hypothesis}{Hypothesis}
\crefname{hypothesis}{Hypothesis}{Hypotheses}
\newsiamthm{claim}{Claim}
\newsiamremark{fact}{Fact}
\crefname{fact}{Fact}{Facts}
\newsiamremark{example}{Example}
\crefname{hypothesis}{Example}{Example}

\headers{On the Moreau envelope properties of weakly convex functions}{M. Renaud, A. Leclaire, and N. Papadakis}

\title{On the Moreau envelope properties of weakly convex functions\thanks{Submitted to the editors DATE.
\funding{This study has been carried out with financial support from the French Direction G\'en\'erale de l’Armement. Experiments presented in this paper were carried out using the PlaFRIM experimental testbed, 
supported by Inria, CNRS (LABRI and IMB), Universite de Bordeaux, Bordeaux INP and Conseil Regional d’Aquitaine (see https://www.plafrim.fr)}}}

\title{
On the Moreau envelope properties of weakly convex functions
}

\author{
  Marien Renaud\thanks{
  Univ. Bordeaux, CNRS, INRIA, Bordeaux INP, IMB, UMR 5251, F-33400 Talence, France (\email{marien.renaud@math.u-bordeaux.fr})}
  \and
   Arthur Leclaire\thanks{LTCI, T\'el\'ecom Paris, IP Paris, France}
   \and
   Nicolas Papadakis\footnotemark[1]}
\usepackage{amsopn}
\begin{document}

\maketitle

\begin{abstract}
    In this document, we present the main properties satisfied by the Moreau envelope of weakly convex functions. The Moreau envelope has been introduced in convex optimization to regularize convex functionals while preserving their global minimizers. However, the Moreau envelope is also defined for the more general class of weakly convex function and can be a useful tool for optimization in this context. The main properties of the Moreau envelope have been demonstrated for convex functions and are generalized to weakly convex function in various works. This document summarizes the vast literature on the properties of the Moreau envelope and provides the associated proofs.
\end{abstract}



\section{Introduction}


This document focuses on functions satisfying (relaxed) convexity assumptions.
\begin{definition}\label{def:cvx_def}
A function $f$ is called
\begin{itemize}
    \item \textit{convex} if and only if  $\forall x,y \in \R^d$, $\forall \lambda \in [0,1]$, 
    \begin{align}\label{eq:convexity}
 f(\lambda x + (1-\lambda)y) \le \lambda f(x) + (1-\lambda)f(y) ,
 \end{align}
    \item \textit{$\rho$-weakly convex}, with $\rho \in \R$, if and only if $f + \frac{\rho}{2}\|\cdot\|^2$ is convex.
    \item \textit{$\mu$-strongly convex}, with $\mu > 0$, if and only if $f$ is $-\mu$-weakly convex.
\end{itemize}
\end{definition}


Weakly convex functions have been first studied in~\cite{nurminskii1973quasigradient}, and are also called semi-convex functions~\cite{crandall1992user, johnston2025performance}. 
The class of weakly convex function is large as it includes convex functions and smooth functions, i.e. functions with Lipschitz gradient. More generally, the authors of~\cite[Lemma 4.2]{drusvyatskiy2019efficiency} show that any function of the form $h \circ c$, with $h$ convex and $c$ a smooth map with Lipschitz Jacobian is weakly convex. Note that such a composition is not necessary convex or smooth.

Many applied optimization problems can be reformulated using weakly convex functions, including phase retrieval~\cite{davis2020nonsmooth}, minimization of the conditional value-at-risk~\cite{rockafellar2000optimization, ben1986expected}, graph synchronization~\cite{abbe2014decoding} or robust principal component analysis~\cite{candes2011robust}. We refer to~\cite{drusvyatskiy2017proximal} for a survey on these applications.
More recently, weakly convex functions play a central role in modern techniques for regularizing inverse problems with deep neural networks~\cite{ebner2025error,goujon2024learning,hurault2024convergent,hurault2022proximal,kamilov2023plug,rafique2022weakly,renaud2025stability,savanier2025learning, Shumaylov24, tan2024provably}. 
Since they are non-convex, minimizing weakly convex functions remains a challenging task. A thorough understanding of their properties is therefore essential for effective optimization, especially when using proximal methods.

The Moreau envelope, also called Moreau-Yosida envelope, was first introduced in the seminal work of J.-J.  Moreau~\cite{moreau1965proximite} to regularize (non-differentiable) convex functions, thereby facilitating their optimization.  
In another context, such an envelope was introduced by K. Yosida in~\cite[Chapter IX, Section 4, page 240]{Yosida1964} to construct smooth approximations of the infinitesimal generator of a semi-group.
Detailed explanation on the Moreau envelope for convex functions can be found in~\cite[Part 1.G]{rockafellar2009variational}. 
The Moreau envelop has been used in many applications for its smoothing properties in optimization~\cite{bohm2021variable,sun2023algorithms,hu2023non,chen2025stochastic} or in langevin diffusion~\cite{pereyra2016proximal,durmus2018efficient,luu2021sampling, renaud2025stability,crucinio2025optimal,habring2025diffusion}.

The Moreau envelope regularizes in the sense that it is always differentiable (Proposition~\ref{prop:nabla_g_weakly_cvx_appendix}) and preserves the minimal value and minimizers of the original function (Proposition~\ref{prop:same_argmin} and Lemma~\ref{lemma:caracterization_critical_point_moreau_env}). However, it is not clear if the Moreau envelope "convexifies" the original function. Some results (Lemma~\ref{lemma:moreau_envelop_weakly_convex}  and Lemma~\ref{lemma:strong_convexity_moreau_envelop}) indeed show that the Moreau envelope degrades both weak and strong convexity constants, whereas another one (Lemma~\ref{lemma:measure_of_convexity}) suggests that the Moreau envelope improves the convexity by non-increasing the non-convexity criteria~\eqref{eq:nc_criteria}.

The results presented in this document are known by experts in the field. However, to our knowledge, there is no single reference that explicitly states and proves all of them, 
especially for {\bf weakly convex functions}. Different proofs can be found in~\cite{bauschke2017correction,gribonval2020characterization, hoheiselproximal, junior2023local, moreau1965proximite, rockafellar2009variational}. 
Building on the work initiated in~\cite[Appendix D]{renaud2025stability}, we provide here proofs that rely on elementary tools. For simplicity, we choose not to introduce the Clarke sub-differential~\cite{clarke1983nonsmooth} which can be used to analyze the Moreau envelope of non-differentiable functions. 
Nevertheless the results involving the differentiability assumption (Lemma~\ref{lemma:more_convex_conjuguate_properties}(ii) and (iii), Proposition~\ref{prop:notation_prox_operator} and Lemma~\ref{lemma:same_critical_points}) can be generalized to non-differentiable functions using the Clarke sub-differential and following the same proof techniques.
We also refer to~\cite{jourani2014differential} for refined results on the regularity of the Moreau envelope in the non-convex and non-smooth setting and to~\cite{hiriart2021regularisation} for a pedagogical introduction to the Moreau envelope in the convex case.

\paragraph{Overview of the document}
First, in section~\ref{sec:def}, we define the Moreau envelope $g^\gamma$ and give some first immediate properties. In section~\ref{sec:dependence_gamma_moreau_envelop}, we analyze the dependency of $g^\gamma$ in the parameter $\gamma$. In section~\ref{sec:link_moreau_conjugate}, we show the intrinsic link between the Moreau envelope and the convex conjugate, as well as a duality result (Proposition~\ref{prop:dual_result_prox}). In section~\ref{sec:cvx_weakl_cvx_properties_moreau}, we investigate how the Moreau envelope preserves the convexity, the weak convexity and the strong convexity. In section~\ref{sec:properties_prox_operator}, we derive the main properties of the proximal operator for weakly convex functions. In section~\ref{sec:moreau_diff_optim}, we prove that the Moreau envelope preserves critical points and minimizers, with the advantage of being differentiable. Finally, in section~\ref{sec:second_derivation_mro_env_imgae_prox}, the second derivative of the Moreau envelope is analyzed and we prove that the image of the proximal operator is almost convex.

\section{Definitions and first properties}\label{sec:def}
In this section, we present the hypotheses satisfied by the functions under consideration and recall some of their useful properties. We then introduce the Moreau envelope and the proximal operator, which are the main technical tools used throughout this document.

For a function $f : \R^d \mapsto \R \cup \{+\infty\}$, we define $\text{dom}(f) = \{x \in \R^d | f(x) \in \R\}$. A function $f$ is said to be \textit{proper} if and only if $\text{dom}(f) \neq \emptyset$. A function $f$ is said to be \textit{lower semicontinuous} if and only if $\forall x \in \R^d, \liminf_{y \to x}{f(y)} \ge f(x)$. In the following, we will suppose that the functions are proper and lower semicontinuous.

We will mainly focus on the class of weakly convex functions, which notably includes the class of functions with Lipschitz continuous gradients. Indeed, any $L$-smooth function $f$, \textit{i.e.} such that $\forall x, y \in \R^d$, $\|\nabla f(x) - \nabla f(y)\| \le L \|x -y \|$, is $L$-weakly convex. 

Let us now recall some useful properties satisfied by functions  satisfying relaxed convexity assumptions. 
As a generalization of the convexity inequality~\eqref{eq:convexity}, if $f$ is $\rho$-weakly convex then $\forall x, y \in \R^d$, $\forall \lambda \in [0,1]$, we have
\begin{align}\label{eq:ineq_wk_cvx}
    f(\lambda x + (1 - \lambda)y) \le \lambda f(x) + (1 - \lambda)f(y) + \frac{\rho}{2} \lambda (1 - \lambda)\|x - y\|^2.
\end{align}

Next, if $f$ is $\rho$-weakly convex and differentiable, then $\forall x, y \in \R^d$
\begin{align}\label{eq:ineq_wk_cvx_diff}
    \langle x - y, \nabla f(x) - \nabla f(y) \rangle \ge - \rho \|x - y\|^2.
\end{align}

Finally, if  a function $f$ is $\rho$-weakly convex and twice differentiable then $\forall x, y \in \R^d$, 
\begin{align}\label{eq:ineq_wk_cvx_2_diff}
    x^T \nabla^2 f(y) x \ge -\rho \|x\|^2.
\end{align}
Due to Definitions~\ref{def:cvx_def}, note that convex functions are $0$-weakly convex functions and $\mu$-strongly convex functions are $-\mu$-weakly convex ones. Therefore, the classes of convex and strongly convex functions are  included in the class of weakly convex functions,  and the three above inequalities hold for $\rho=0$ (convex functions) or $\rho=-\mu$ ($\mu$-strongly convex functions).

We now introduce the notion of inf-convolution~\cite{moreau1970inf} as well as the Moreau envelope~\cite{moreau1965proximite}.

\begin{definition}
The inf-convolution between two functions $f, g : \R^d \to \R$ is defined for $x \in \R^d$ by
\begin{align}\label{eq:inf_conv}
    (f \square g)(x) = \inf_{y \in \R^d} f(y) + g(x-y),
\end{align}
with the convention $(f \square g)(x) = -\infty$ if the right-hand-side of equation~\eqref{eq:inf_conv} is not bounded from below.
\end{definition}

\begin{definition}
For $\gamma > 0$, the Moreau envelope of $g : \R^d \to \R$, noted $g^{\gamma}$ is defined for $x \in \R^d$ by
\begin{align}\label{eq:moreau_env}
    g^{\gamma}(x) = \left(g \square \frac{1}{2\gamma} \|\cdot\|^2\right)(x) = \inf_{y \in \R^d} \frac{1}{2\gamma} \|x - y\|^2 + g(y).
\end{align}
\end{definition}
The Moreau envelope is finite on $\R^d$ if $g$ is $\rho$-weakly convex, \textit{i.e.} $\rho > 0$ and $g +\frac{\rho}{2}\|\cdot\|^2$ is convex, with $\gamma \rho < 1$. In fact, the function $y \mapsto \frac{1}{2\gamma} \|x - y\|^2 + g(y)$ is $\left(\frac{1}{\gamma} - \rho \right)$ strongly convex so it admits a unique minimum for $\gamma \rho < 1$. The minimal point is called the proximal point 
\begin{align}\label{eq:prox}
\Prox{\gamma g}{x}= \argmin_{y \in \R^d} \frac{1}{2\gamma} \|x - y\|^2 + g(y). 
\end{align}
We thus have
\begin{align}\label{eq:moreau_prox}
g^\gamma(x)=g(\Prox{\gamma g}{x}) + \frac{1}{2\gamma} \|x - \Prox{\gamma g}{x}\|^2.
\end{align}

The Moreau envelope is an approximation of $g$ in the sense that $\forall x \in \R^d$, $\lim_{\gamma \to 0} g^\gamma(x) = g(x)$ (see Proposition~\ref{prop:moreau_envelop_dependence_in_gamma}) and it is a regularization as $g^\gamma$ is differentiable (see Proposition~\ref{prop:nabla_g_weakly_cvx_appendix}) even if $g$ is not.

It is important to note that $g$ and $g^\gamma$ have the same minimal value as
\begin{align*}
    \inf_{x \in \R^d} g(x) = \inf_{x, y \in \R^d} g(x) + \frac{1}{2\gamma} \|x - y\|^2 = \inf_{y \in \R^d} g^\gamma(y).
\end{align*}
We will also show in Proposition~\ref{prop:same_argmin} that $g$ and $g^\gamma$ have the same minimizers. Therefore, the Moreau envelope is particularly interesting for optimization. We can minimize the differentiable function $g^\gamma$ instead of the possibly non-differentiable $g$~\cite{yamada2011minimizing}.

\begin{example}\label{ex1}
For $g(x) =  -\frac{\rho}{2} \|x\|^2$ with $\rho \in \R$, and $\gamma > 0$ such that $\gamma \rho < 1$, we have $g^{\gamma}(x) = -\frac{\rho}{2(1 - \gamma \rho)}\|x\|^2$.

In particular, if $\mu = -\rho > 0$ and $g(x) = \frac{\mu}{2}\|x\|^2$, then $\forall \gamma > 0$, we have $g^{\gamma}(x) = \frac{\mu}{2(1 - \gamma \mu)}\|x\|^2$.
\end{example}

\section{On the dependency of \texorpdfstring{$g^\gamma$}{g gamma} w.r.t. \texorpdfstring{$\gamma > 0$}{gamma positive}}\label{sec:dependence_gamma_moreau_envelop}
The following result shows the monotonicity of $g^\gamma$ in $\gamma$ and the fact that the Moreau envelope is an approximation of $g$ when $\gamma \to 0$. Moreover, we show that $g^\gamma$ is differentiable w.r.t. $\gamma$ and gives the value of the derivative.

First, to analyze this dependency, we need a technical lemma, that will be useful for other properties (see Lemma~\ref{lemma:expansivity_prox}).
\begin{lemma}\label{lemma:technical_lemma_prox}
If $g$ is $\rho$-weakly convex function with $\gamma \rho < 1$, for $x, y \in \R^d$ and $p = \textsf{Prox}_{\gamma g}(x)$, we have
\begin{align*}
    g(p) + \frac{1}{\gamma} \langle x- p , y - p  \rangle \le g(y) + \frac{\rho}{2} \|y - p\|^2.
\end{align*}
\end{lemma}

This lemma is a generalization of~\cite[Proposition 2.26]{bauschke2017correction} proposed in~\cite[Lemma 3.2]{hoheiselproximal}.

\begin{proof}
For $x, y \in \R^d$, we denote by $p = \Prox{\gamma g}{x}$ and, with $\alpha \in (0,1)$, $p_{\alpha} = \alpha y + (1-\alpha) p$. By definition of the proximal operator~\eqref{eq:prox}, we have
\begin{align*}
    g(p) + \frac{1}{2\gamma} \|x - p\|^2 \le g(p_\alpha) + \frac{1}{2\gamma} \|x - p_\alpha\|^2.
\end{align*}
By the previous inequality and the $(\frac1\gamma-\rho)$-strong convexity of $g+\frac1{2\gamma}\|\cdot\|^2$ in equation~\eqref{eq:ineq_wk_cvx} for $\gamma\rho<1$, we get
\begin{align*}
    &g(p) \le g(p_\alpha) + \frac{1}{2\gamma} \|x - p_\alpha\|^2 - \frac{1}{2\gamma} \|x - p\|^2 \\
    &\le \alpha g(y) + (1-\alpha) g(p) + \frac{\alpha}{2\gamma}  \|x-y\|^2 +\frac{1-\alpha}{2\gamma} \|x-p\|^2 \\
    &-\left(\frac1{\gamma}-\rho\right)\frac{\alpha(1-\alpha)}2\|y-p\|^2- \frac{1}{2\gamma} \|x - p\|^2\\
    &\le \alpha g(y) + (1-\alpha) g(p) - \frac{\alpha}{\gamma} \langle x-p, y-p \rangle + \frac{\alpha^2}{2\gamma} \|y-p\|^2+ \frac{\rho}{2}\alpha (1-\alpha) \|y-p\|^2.
\end{align*}
By rearranging the terms, we get
\begin{align*}
    \alpha \langle x-p, y-p \rangle \le \alpha \gamma g(y) - \alpha \gamma g(p)  + \frac{\alpha}{2} (\alpha(1 - \rho \gamma) +\rho\gamma) \|y-p\|^2.
\end{align*}
By dividing by $\alpha$ and taking $\alpha \to 0$, we get
\begin{align*}
     \langle x-p, y-p \rangle \le \gamma g(y) - \gamma g(p)  + \frac{\gamma \rho}{2} \|y-p\|^2.
\end{align*}
\end{proof}

\begin{proposition}\label{prop:moreau_envelop_dependence_in_gamma}
If $g$ is $\rho$-weakly convex, then for $x\in\R^d$,  $\gamma \in (0, \frac{1}{\rho})\to g^\gamma(x)$ is non-increasing and 
\begin{align}\label{eq:lim_env}
    \lim_{\gamma \to 0} g^\gamma(x) = g(x).
\end{align}
Moreover,  $\gamma \in (0,\frac{1}{\rho}) \to g(\Prox{\gamma g}{x})$ is non-increasing and 
\begin{align}
    \lim_{\gamma \to 0} \Prox{\gamma g}{x} = x,~~~~ \lim_{\gamma \to 0} g(\Prox{\gamma g}{x}) = g(x).
\end{align}
\end{proposition}
Proposition~\ref{prop:moreau_envelop_dependence_in_gamma} shows that $g^\gamma$ has a monotonic dependency in $\gamma$. Moreover, when the step-size $\gamma \to 0$, the different objects have a well defined and computable limit. Note that in particular, we have $\forall \gamma \in (0, \frac{1}{\rho}), \forall x \in \R^d, \inf g \le g^{\gamma}(x) \le g(x)$.
\begin{proof}
We need $\gamma < \frac{1}{\rho}$ to ensure that the Moreau envelope is well defined. We obtain from the definition of the Moreau envelope (see relation~\eqref{eq:moreau_env}) that $\gamma \to g^\gamma(x)$ is non-increasing and  $g^\gamma(x) \le g(x)$. The rest of the proof is a generalization of~\cite[Proposition 12.33]{bauschke2017correction} in the weakly convex case.

By applying Lemma~\ref{lemma:technical_lemma_prox}, with $p = \textsf{Prox}_{\gamma g}(x)$, we have
\begin{align*}
    \langle x-p, y-p \rangle \le \gamma g(y) - \gamma g(p)  + \frac{\gamma \rho}{2} \|y-p\|^2.
\end{align*}
For $\frac1\rho>\mu \ge \gamma$, and $q = \Prox{\mu g}{x}$. Lemma~\ref{lemma:technical_lemma_prox} gives that $\forall y\in\R^d$:
\begin{align*}
     \langle x-q, y-q \rangle \le \gamma g(y) - \gamma g(q)  + \frac{\mu \rho}{2} \|y-q\|^2.
\end{align*}
Taking respectively $y=q$ and $y=p$ in the two previous inequalities, we obtain
\begin{align*}
    \langle q-p, x-p \rangle &\le \gamma g(q) - \gamma g(p)  + \frac{\gamma \rho}{2} \|q-p\|^2 \\
     \langle p-q, x-q \rangle &\le \mu g(p) - \mu g(q)  + \frac{\mu \rho}{2} \|p-q\|^2.
\end{align*}
By summing these two relations and rearranging the terms, we get
\begin{align*}
    \left( 1 - \frac{\rho}{2}(\mu + \gamma) \right) \|p - q\|^2 \le (\mu - \gamma) \left(g(p) - g(q) \right).
\end{align*}
Therefore, if $\gamma, \mu \in (0,\frac{1}{\rho})$, for $\mu \ge \gamma$, we get $g(\Prox{\mu g}{x}) \le g(\Prox{\gamma g}{x})$. It proves that $\gamma \in (0,\frac{1}{\rho}) \to g(\Prox{\gamma g}{x})$ is non-increasing.

We now show that $g^\gamma(x) \to g(x)$ when $\gamma \to 0$. 
We denote $M = \sup_{\gamma \in (0, \frac{1}{2\rho}]}g^\gamma(x) < +\infty$. From the characterization~\eqref{eq:moreau_prox} we have $\forall \gamma \in (0, \frac{1}{2\rho}]$, 
\begin{align*}
    M \ge g^\gamma(x)=g(\Prox{\gamma g}{x}) + \frac{1}{2\gamma} \|x - \Prox{\gamma g}{x}\|^2.
\end{align*}
Since, $\phi : y \mapsto g(y) + \rho \|x-y\|^2$ is $\rho$-strongly convex thus coercive and given that $\phi(\Prox{\gamma g}{x})$ has a value lower than $M$, we get that $m = \sup_{\gamma \in (0, \frac{1}{2\rho}]}\|\Prox{\gamma g}{x}\| < +\infty$. 
Moreover, the function $g +\frac{\rho}{2}\|\cdot\|^2$ is convex and lower bounded by an affine function. So there exist $a_1 \in\R^d$ and $a_2 \in \R$, such that $\forall x \in \R^d$, $g(x) +\frac{\rho}{2}\|x\|^2 \ge \langle a_1, x \rangle + a_2$. Combining the previous facts, we get
\begin{align*}
    M &\ge g(\Prox{\gamma g}{x}) + \frac{1}{2\gamma} \|x - \Prox{\gamma g}{x}\|^2 \\
    &\ge \langle a_1, \Prox{\gamma g}{x} \rangle + a_2 - \frac{\rho}{2}\|\Prox{\gamma g}{x}\|^2 + \frac{1}{2\gamma} \|x - \Prox{\gamma g}{x}\|^2 \\
    &\ge - \| a_1\| m + a_2 - \frac{\rho}{2}m^2 + \frac{1}{2\gamma} \|x - \Prox{\gamma g}{x}\|^2.
\end{align*}
So we get $\lim_{\gamma \to 0}\|x - \Prox{\gamma g}{x}\| = 0$. Finally, by the lower semi-continuity of $g$, we have
\begin{align*}
    g(x) \ge \lim_{\gamma \to 0} g^\gamma(x) \ge \lim_{\gamma \to 0} g(\Prox{\gamma g}{x}) \ge g(x),
\end{align*}
which proves the limit when $\gamma \to 0$.
\end{proof}

The previous proposition can be completed by the following proposition.

\begin{proposition}\label{prop:derivative_in_gamma}
For $g$ a $\rho$-weakly convex function, the function $(\gamma, x) \in [0, \frac{1}{\rho}) \times \R^d \mapsto \Prox{\gamma g}{x}$ is continuous and the function $(\gamma, x) \in (0, \frac{1}{\rho}) \times \R^d \mapsto g^\gamma(x)$ is lower semi-continuous.
Moreover the function $\gamma \in (0, \frac{1}{\rho}) \mapsto g^\gamma(x)$ is differentiable and
\begin{align*}
    \frac{\partial g^\gamma}{\partial \gamma}(x) = -\frac{1}{\gamma^2}\|x - \Prox{\gamma g}{x}\|^2.
\end{align*}
\end{proposition}
Proposition~\ref{prop:derivative_in_gamma} can be found in~\cite[Lemma 2 and Lemma 6]{habring2025diffusion}. Note that if $g$ is assumed to be continuous, we can prove that the function $(\gamma, x) \in [0, \frac{1}{\rho}) \times \R^d \mapsto g^\gamma(x)$ is globally continuous by adapting the proof with the continuity extension theorem.
\begin{proof}
For a sequence $(\gamma_n, x_n) \in \left((0, \frac{1}{\rho}) \times \R^d\right)^\N$ such that $(\gamma_n, x_n) \to (\gamma, x) \in (0, \frac{1}{\rho}) \times \R^d$, by the definition of the Moreau envelop~\eqref{eq:moreau_env}, we have
\begin{align*}
    \frac{1}{2\gamma_n}\|x_n - \Prox{\gamma_n g}{x_n}\|^2 + g(\Prox{\gamma_n g}{x_n}) \le \frac{1}{2\gamma_n}\|x_n\|^2 + g(0).
\end{align*}
Because the sequences $x_n$ and $\gamma_n$ have a limit, there exists $M \ge 0$ such that $\forall n \in \N$, $\frac{1}{2\gamma_n}\|x_n\|^2 + g(0) \le M$. There also exists $\bar{\gamma} < \frac{1}{\rho}$ such that $\forall n \in \N$, $\gamma_n \le \bar{\gamma}$, so we get
\begin{align*}
    \frac{1}{2\bar{\gamma}}\|x_n - \Prox{\gamma_n g}{x_n}\|^2 + g(\Prox{\gamma_n g}{x_n}) \le M.
\end{align*}
However we have 
\begin{align*}
    &\|x_n - \Prox{\gamma_n g}{x_n}\|^2 = \|x_n\|^2 + \|\Prox{\gamma_n g}{x_n}\|^2 - 2 \langle x_n, \Prox{\gamma_n g}{x_n} \rangle \\
    &\ge \|\Prox{\gamma_n g}{x_n}\|^2 - 2 \| x_n \| \|\Prox{\gamma_n g}{x_n} \|
\end{align*}
As the sequence $x_n$ has a limit, there exists $M_x \ge 0$, such that $\|x_n\| \le M_x$. Then we have  
\begin{align*}
    \|x_n - \Prox{\gamma_n g}{x_n}\|^2 \ge \|\Prox{\gamma_n g}{x_n}\|^2 - 2 M_x \|\Prox{\gamma_n g}{x_n} \|.
\end{align*}
By injecting this inequality in the previous equation, we get
\begin{align*}
    \frac{1}{2\bar{\gamma}} \|\Prox{\gamma_n g}{x_n}\|^2 + g(\Prox{\gamma_n g}{x_n}) - \frac{1}{\bar{\gamma}} M_x \|\Prox{\gamma_n g}{x_n} \| \le M.
\end{align*}
However, because $g$ is $\rho$-weakly convex and $\bar{\gamma} < \frac{1}{\rho}$, the function $y \mapsto \frac{1}{2\bar{\gamma}} \|y\|^2 + g(y) - \frac{1}{\bar{\gamma}} M_x \|y\|$ is strongly convex, thus coercive. Due to the previous inequality, there exists $M_p \ge 0$, such that $\forall n \in \N$, $\|\Prox{\gamma_n g}{x_n}\| \le M_p$. By taking a subsequence, we deduce (without re-indexing for simplicity) that there exists $p \in \R^d$ such that $\lim_{n \to +\infty} \Prox{\gamma_n g}{x_n} = p$.

By the lower semi-continuity of $g$, we have that for any $y \in \R^d$
\begin{align*}
    \frac{1}{2 \gamma}\|x - p\|^2 + g(p) &\le \lim_{n \to +\infty} \frac{1}{2 \gamma_n}\|x_n - \Prox{\gamma_n g}{x_n}\|^2 + g(\Prox{\gamma_n g}{x_n}) \\
    &\le \lim_{n \to +\infty} \frac{1}{2 \gamma_n}\|x_n - y\|^2 + g(y) \\
    &= \frac{1}{2 \gamma}\|x - y\|^2 + g(y).
\end{align*}
The previous inequality implies that $p = \Prox{\gamma g}{x}$. So $\Prox{\gamma g}{x}$ is the unique accumulation point of $\Prox{\gamma_n g}{x_n}$, so we get that $\lim_{n \to +\infty} \Prox{\gamma_n g}{x_n} = \Prox{\gamma g}{x}$. This proves that $(\gamma, x) \in (0, \frac{1}{\rho})\times \R^d \mapsto \Prox{\gamma g}{x}$ is continuous. 
Moreover, combining Proposition~\ref{prop:moreau_envelop_dependence_in_gamma} and Lemma~\ref{lemma:prox_lipschitz_weakly_cvx}, we obtain that $\lim_{n \to +\infty} \Prox{\gamma_n g}{x_n} = x$, if $x_n \to x$ and $\gamma_n \to 0$. So, we get that $(\gamma, x) \in [0, \frac{1}{\rho})\times \R^d \mapsto \Prox{\gamma g}{x}$ is continuous.

Finally, from the lower semi-continuity of $g$, we get
\begin{align*}
    &\lim_{n \to +\infty} \frac{1}{2 \gamma_n}\|x_n - \Prox{\gamma_n g}{x_n}\|^2 + g(\Prox{\gamma_n g}{x_n}) \\
    &\ge \frac{1}{2 \gamma}\|x - \Prox{\gamma g}{x}\|^2 + g(\Prox{\gamma_n g}{x}).
\end{align*}
This shows that the function $(\gamma, x) \in (0, \frac{1}{\rho})\times \R^d \mapsto g^\gamma(x)$ is lower semi-continuous.

We now prove that $g^\gamma$ is differentiable w.r.t. $\gamma$. For $0 < \gamma < \lambda < \frac{1}{\rho}$, by the definition of the Moreau envelop~\eqref{eq:moreau_env}, we get  for $x \in \R^d$ 
\begin{align*}
    &\frac{1}{2\lambda}\|x - \Prox{\lambda g}{x}\|^2 + g(\Prox{\lambda g}{x}) - \frac{1}{2\gamma} \|x - \Prox{\lambda g}{x}\|^2 - g(\Prox{\lambda g}{x}) \\
    &\le g^\lambda(x) - g^\gamma(x) \\
    &\le \frac{1}{2\lambda}\|x - \Prox{\gamma g}{x}\|^2 + g(\Prox{\gamma g}{x})- \frac{1}{2\gamma} \|x - \Prox{\gamma g}{x}\|^2 - g(\Prox{\gamma g}{x}).
\end{align*}
By re-arranging the terms, we get that
\begin{align*}
    -\frac{1}{2\lambda\gamma} \|x - \Prox{\lambda g}{x}\|^2 \le \frac{g^\lambda(x) - g^\gamma(x)}{\lambda - \gamma} \le -\frac{1}{2\lambda\gamma} \|x - \Prox{\gamma g}{x}\|^2.
\end{align*}
By the continuity of $(\gamma,x) \mapsto  \Prox{\gamma g}{x}$, it prove that $\gamma \in (0, \frac{1}{\rho}) \mapsto g^\gamma(x)$ is differentiable with the value
\begin{align*}
    \frac{\partial g^\gamma}{\partial \gamma}(x) = -\frac{1}{\gamma^2}\|x - \Prox{\gamma g}{x}\|^2.
\end{align*}
\end{proof}

\section{The intrinsic link between the Moreau envelope and the convex conjugate}\label{sec:link_moreau_conjugate}
In this part, we introduce the convex conjugate and study its relation with the Moreau envelope. In particular, Proposition~\ref{prop:dual_result_prox} shows a duality result between the convex conjugate and the Proximal operator.

\begin{definition}
The convex conjugate of a proper function $f:\R^d \to \R$, denoted by $f^\star$, is defined for $x \in \R^d$ by
\begin{align*}
    f^\star(x) = \sup_{y \in \R^d} \langle x, y \rangle - f(y).
\end{align*}
\end{definition}

The notion of convex conjugate of $f$ is related with the inf-convolution and the Moreau envelope, as shown by the following results from Lemma~\ref{lemma:fondamental_properties_convex_conjuguate} to Proposition~\ref{prop:dual_result_prox}.

\begin{lemma}\label{lemma:fondamental_properties_convex_conjuguate}
For $f: \R^d \to \R$ a proper function, we have the following properties
\begin{enumerate}[label=(\roman*)]
    \item  $f^\star$ is convex.
    \item For $f, g: \R^d \to \R$ convex, then $(f \square g)^\star=f^\star+g^\star$.
    \item $f : \R^d \to \R^d$ is convex if and only if $f^{\star \star} = f$.
    \item $\left(\frac{\alpha}{2}\|\cdot\|^2\right)^\star= \frac{1}{2\alpha}\|\cdot\|^2$.
\end{enumerate}
\end{lemma}
By Lemma~\ref{lemma:fondamental_properties_convex_conjuguate}(i), $f^\star$ is convex even if $f$ is non-convex. Moreover,   Lemma~\ref{lemma:fondamental_properties_convex_conjuguate}(ii) shows the compatibility of the convex conjugate with the inf-convolution operation.

\begin{proof}
\textbf{(i)}
For $x, y \in \R^d$ and $\lambda \in [0, 1]$, we have
\begin{align*}
    f^\star(\lambda x + (1-\lambda)y) &= \sup_{z \in \R^d} \langle \lambda x + (1-\lambda)y, z \rangle - f(z) \\
    &= \sup_{z \in \R^d} \lambda \left(\langle x , z \rangle - f(z)\right) + (1-\lambda)\left(\langle y , z \rangle - f(z)\right) \\
    &\le \lambda \sup_{z \in \R^d}\left(\langle x , z \rangle - f(z)\right) + (1-\lambda)  \sup_{z \in \R^d}\left(\langle y , z \rangle - f(z)\right) \\
    &\le \lambda f^\star(x) + (1-\lambda)f^\star(y),
\end{align*}
which shows the convexity of $f^\star$.

\textbf{(ii)} For $x \in \R^d$, 
\begin{align*}
    (f \square g)^\star(x) &= \sup_{y \in \R^d} \langle x, y\rangle - \left( f \square g \right)(y) \\
    &= \sup_{y \in \R^d} \langle x, y\rangle - \inf_{z \in \R^d} f(z) + g(y-z) \\
    &= \sup_{y \in \R^d} \sup_{z \in \R^d} \langle x, y\rangle -f(z) - g(y-z) \\
    &= \sup_{z \in \R^d} -f(z) + \sup_{y \in \R^d} \langle x, y\rangle - g(y-z) \\
    &= \sup_{z \in \R^d} -f(z) + \langle x, z\rangle + \sup_{y \in \R^d} \langle x, y - z\rangle - g(y-z) \\
    &= \sup_{z \in \R^d} -f(z) + \langle x, z\rangle + g^\star(x) \\
    &= f^\star(x) + g^\star(x).
\end{align*}

\textbf{(iii)}
If $f^{\star \star} = f$, then by point (i), $f$ is convex.

If $f$ is convex, we have
\begin{align*}
    f(x) = \sup_{(a, b) \in \Sigma} \langle a, x \rangle + b,
\end{align*}
with $\Sigma = \{ (a, b) \in \R^d \times \R\,|\, \forall x \in \R^d, \langle a, x \rangle + b \le f(x) \}$. Hence $(a,b) \in \Sigma$ if and only if $\forall x \in \R^d$, $\langle a, x \rangle - f(x) \le -b$, \textit{i.e.} $-b \ge f^\star(a)$. Then, we have
\begin{align*}
    f(x) &= \sup_{(a, b) \in \Sigma} \langle a, x \rangle + b = \sup_{a \in \R^d, b \le -f^\star(a)} \langle a, x \rangle + b \\
    &= \sup_{a \in \R^d} \langle a, x \rangle -f^\star(a) = f^{\star\star}(x).
\end{align*}

\textbf{(iv)} For $x\in \R^d$, $\left(\frac{\alpha}{2}\|x\|^2\right)^\star =\sup_{y \in \R^d} \langle x, y \rangle - \frac{\alpha}{2}\|y\|^2$. The sup is reached at $y = \frac{x}{\alpha}$, so that $\left(\frac{\alpha}{2}\|\cdot\|^2\right)^\star= \frac{1}{2\alpha}\|\cdot\|^2$.

\end{proof}

\begin{lemma}\label{lemma:more_convex_conjuguate_properties}
For $f: \R^d \to \R$ a proper function, we have the following properties
\begin{enumerate}[label=(\roman*)]
\item $\forall x, y \in \R^d, f(x) + f^\star(y) \ge \langle x, y \rangle.$
\item For $f$ convex and differentiable and $x \in \R^d$, with $y = \nabla f(x)$, we have $f(x) + f^\star(y) = \langle x, y \rangle$.
\item For $f$ convex and differentiable, we have $y = \nabla f(x)$ if and only if $x = \nabla f^\star(y)$.
\item For $f$ convex and differentiable, $f$ is $\mu$-strongly convex if and only if $f^\star$ is $1/\mu$-smooth.
\end{enumerate}
\end{lemma}
Lemma~\ref{lemma:more_convex_conjuguate_properties}(iii) can be reformulated as $\nabla f \circ \nabla f^\star = I_d$.
\begin{proof}
\textbf{(i)} For $x, y \in \R^d$
\begin{align*}
    f(x) + f^\star(y) = \sup_{z \in \R^d} \langle y,z\rangle + f(x) - f(z) \ge \langle y,x\rangle.
\end{align*}

\textbf{(ii)} If $y = \nabla f(x)$, then by the convexity of $f$, we get, $\forall z \in \R^d$,
\begin{align*}
    f(z) &\ge f(x) + \langle y, z-x\rangle \\
    \langle x, y \rangle - f(x) &\ge \langle y, z\rangle - f(z)\\
    \langle x, y \rangle - f(x) &\ge f^\star(y).
\end{align*}
So, by definition of $f^\star$,  we get $f^\star(y) = \langle x, y \rangle - f(x)$ .

\textbf{(iii)} If $y = \nabla f(x)$, by the point (ii), we have $f(x) + f^\star(y) = \langle x, y \rangle$. Then for all $t \in \R^d$
\begin{align*}
    f^\star(t) &= \sup_{s \in \R^d} \langle t, s \rangle - f(s)\\
    &\ge \langle t, x \rangle - f(x) \\
    &\ge \langle t - y, x \rangle + \langle y, x \rangle - f(x) \\
    &\ge \langle t - y, x \rangle + f^\star(y).
\end{align*}
Thanks to Lemma~\ref{lemma:fondamental_properties_convex_conjuguate}(i) $f^\star$ is convex. Hence we get that $x = \nabla f^\star(y)$. The reverse is true because $f$ is convex so $f^{\star\star} = f$ by Lemma~\ref{lemma:fondamental_properties_convex_conjuguate}(iii). Thus, we can apply the same reasoning with $f^\star$ instead of $f$.

\textbf{(iv)} We first prove the direct implication. We denote $u(x,y) = \langle x, y \rangle - f(y)$. Thanks to equation~\eqref{eq:ineq_wk_cvx_diff}, the $\mu$-strong convexity of $f$ implies that $\forall x, y \in \R^d$, 
\begin{equation}\label{eq:strong_conv}
\langle \nabla f(x) - \nabla f(y), x-y \rangle \ge \mu \|x-y\|^2.
\end{equation} Thus $\nabla f : \R^d \to \R^d$ is injective. As $u$ is strongly concave w.r.t. $y$, it admits a unique maximal point denoted by $y_0$. On this point the optimal condition gives $\nabla_y u(x,y_0) = 0$, so $\nabla f(y_0) = x$. Therefore, $\nabla f$ is surjective, so it is a bijective function. Using points (ii) and (iii) we get
\begin{equation}\label{eq:charac_fstar}
    f^\star(x) = \langle x, (\nabla f)^{-1}(x) \rangle - f((\nabla f)^{-1}(x) ).
\end{equation} 
By applying  Equation~\eqref{eq:strong_conv} on points mapped by  $\left(\nabla f \right)^{-1}$, we get $\forall x, y \in \R^d$,
\begin{align*}
    \langle x-y, (\nabla f)^{-1}(x) - (\nabla f)^{-1}(y) \rangle \ge \mu \|(\nabla f)^{-1}(x) - (\nabla f)^{-1}(y)\|^2.
\end{align*}
The Cauchy-Schwarz inequality then gives
\begin{align}\label{eq:control_inverse_nabla_f}
    \|(\nabla f)^{-1}(x) - (\nabla f)^{-1}(y)\| \le \frac{1}{\mu}\| x-y \|.
\end{align}
Now, we prove that $f^\star$ is differentiable and smooth.

Using~\eqref{eq:charac_fstar}, we have for $x, h \in \R^d$
\begin{align*}
    &f^\star(x+h) - f^\star(x) \\
    &= \langle x + h, (\nabla f)^{-1}(x + h) \rangle - f((\nabla f)^{-1}(x +h) ) - \langle x, (\nabla f)^{-1}(x) \rangle \\
    &+ f((\nabla f)^{-1}(x) ) \\
    &=\langle x, (\nabla f)^{-1}(x + h) - (\nabla f)^{-1}(x) \rangle + \langle h, (\nabla f)^{-1}(x + h) \rangle \\
    &- f((\nabla f)^{-1}(x +h) ) + f((\nabla f)^{-1}(x) ).
\end{align*}
By denoting $u = (\nabla f)^{-1}(x +h) - (\nabla f)^{-1}(x)$, we know that $u  = \mathcal{O}(\|h\|)$ thanks to Equation~\eqref{eq:control_inverse_nabla_f}. Thus, we get
\begin{align*}
    &f^\star(x+h) - f^\star(x) \\
    &= \langle x, u \rangle + \langle h, (\nabla f)^{-1}(x) \rangle  + \langle h, u \rangle - f((\nabla f)^{-1}(x) + u) + f((\nabla f)^{-1}(x) )\\
    &= \langle x, u \rangle + \langle h, (\nabla f)^{-1}(x) \rangle  - \langle \nabla f((\nabla f)^{-1}(x)), u \rangle + o(\|h\|) \\
    &= \langle h, (\nabla f)^{-1}(x) \rangle + o(\|h\|).
\end{align*}
Therefore $f^\star$ is differentiable and $\nabla f^{\star} = (\nabla f)^{-1}$. Combining this result with equation~\eqref{eq:control_inverse_nabla_f} proves that $f^\star$ is $\frac{1}{\mu}$-smooth.

Conversely, if $f^\star$ is $\frac{1}{\mu}$-smooth, since $f^\star$ is convex by the Baillon–Haddad theorem~\cite[Corollary 18.17]{bauschke2017correction}, $\nabla f^\star$ is $\mu$-cocoercive, \textit{i.e.} $\forall x, y \in \R^d$,
\begin{equation}
    \langle \nabla f^\star(x) - \nabla f^\star(y), x - y \rangle \ge \mu \|\nabla f^\star(x) - \nabla f^\star(y)\|^2.
\end{equation}
By denoting $z = \nabla f^\star(x)$ and $t = \nabla f^\star(y)$, by Lemma~\ref{lemma:more_convex_conjuguate_properties}(iii), we have $x = \nabla f(z)$ and $y = \nabla f(t)$, so
\begin{align*}
    \langle z - t, \nabla f(z) - \nabla f(t) \rangle \ge \mu \|z - t\|^2,
\end{align*}
which implies that $f$ is $\mu$-strongly convex, thanks to equation~\eqref{eq:ineq_wk_cvx_diff}.

\end{proof}

\begin{proposition}\label{prop:dual_result_prox}
    For $f$ convex and $x, y, z \in \R^d$, the two following statements are equivalent
    
    \textbf{(i)} $z = x+y$ and $f(x) + f^\star(y) = \langle x, y \rangle$.
    
    \textbf{(ii)} $x = \Prox{f}{z}$ and $y = \Prox{f^\star}{z}$.
\end{proposition}
The previous result has been first proved in~\cite[Proposition 4.a]{moreau1965proximite}.
Proposition~\ref{prop:dual_result_prox} shows the intrinsic link between the proximal operator and the convex conjugate. In particular, it demonstrates that the inequality of Lemma~\ref{lemma:more_convex_conjuguate_properties}(i) is attained only for pairs $(x,y)\in\R^d\times\R^d$ that are of the form $(\Prox{f}{z}, \Prox{f^\star}{z})$. Moreover, it also shows  the equality $\Prox{f}{z} + \Prox{f^\star}{z} = z$. 

\begin{proof}
\textbf{(i) $\implies$ (ii)}
For $u \in \R^d$, by the convex conjugate definition, we have
\begin{align*}
    f^\star(y) \ge \langle u, y \rangle - f(u).
\end{align*}
Using that $f(x) + f^\star(y) = \langle x, y \rangle$, we get
\begin{align*}
    \langle x, y \rangle - f(x) \ge \langle u, y \rangle - f(u).
\end{align*}
From the previous inequality, we obtain for $z = x + y$
\begin{align*}
    \frac{1}{2}\|x-z\|^2 + f(x) &= \frac{1}{2} \|x\|^2 +\frac{1}{2} \|z\|^2 - \langle x, z \rangle + f(x) \\
    &\le \frac{1}{2} \|x\|^2 +\frac{1}{2} \|z\|^2 - \langle x, z \rangle + \langle x, y \rangle + f(u) - \langle u, y \rangle \\
    &\le \frac{1}{2} \|x\|^2 +\frac{1}{2} \|z\|^2 - \|x\|^2 + f(u) - \langle u, z \rangle + \langle u, x \rangle \\
    &\le \frac{1}{2} \|u-z\|^2 + f(u) - \frac{1}{2} \|x-u\|^2.
\end{align*}
Then, necessarily $x$ is the only minimum of the functional $u \mapsto \frac{1}{2} \|u-z\|^2 + f(u)$, which means that $x = \Prox{f}{z}$. A similar computation gives that $y = \Prox{f^\star}{z}$.

\textbf{(ii) $\implies$ (i)}
We note $y' = z-x$, with $x = \Prox{f}{z}$. By the convexity of $f$, for $t \in (0,1)$ and $u \in \R^d$, we have
\begin{align*}
    f(tu+(1-t)x)\le tf(u)+(1-t)f(x).
\end{align*}
Then, because $x = \Prox{f}{z}$, we have $\forall u \in \R^d$ and $\forall t \in (0,1)$,
\begin{align*}
    \frac{1}{2} \|x-z\|^2 + f(x) \le \frac{1}{2}\|tu+(1-t)x - z\|^2 + f(tu+(1-t)x).
\end{align*}
By combining the two previous inequalities, we get
\begin{align*}
    \frac{1}{2} \|x-z\|^2 + f(x) &\le \frac{1}{2}\|tu+(1-t)x - z\|^2 + tf(u)+(1-t)f(x), \\
    -\frac{t^2}{2}\|u-x\|^2 + t \langle u-x, y' \rangle &\le t f(u) - t f(x).
\end{align*}
By dividing by $t$ and letting $t \to 0$, we obtain that for all $u \in \R^d$
\begin{align*}
    \langle u-x, y' \rangle \le f(u) - f(x)\\
    \langle u, y' \rangle - f(u)\le  \langle x, y' \rangle- f(x).
\end{align*}
This implies that
\begin{align*}
    f^\star(y') = \langle x, y' \rangle- f(x),
\end{align*}
which ends the proof.
\end{proof}

\section{On the preservation of the convexity by the Moreau envelope}\label{sec:cvx_weakl_cvx_properties_moreau}
In this part, we prove that the Moreau envelope preserves the convexity (Lemma~\ref{lemma:convexity_moreau_envelop}), weak-convexity (Lemma~\ref{lemma:moreau_envelop_weakly_convex}) and strong-convexity (Lemma~\ref{lemma:strong_convexity_moreau_envelop}) properties. We also show that the Moreau envelope convexifies the original function (Lemma~\ref{lemma:measure_of_convexity}).

We will first need the general result of Lemma~\ref{lemma:inf_convex_functions}.

\begin{lemma}\label{lemma:inf_convex_functions}
If the function $(x,y) \in A \times B \to f(x,y)$, with $A,B$ convex sets, is convex w.r.t. the variable $(x,y) \in A \times B$, then the function $x \in A \mapsto \inf_{y\in B} f(x,y)$ is convex.
\end{lemma}
\begin{proof}
By convexity of $f$, we get that $\forall t \in [0,1]$, $x, x' \in A$, $y,y' \in B$,
\begin{align*}
    f(tx+(1-t)x', ty+(1-t)y') \le t f(x,y) + (1-t) f(x',y').
\end{align*}
By taking the infimum on $y,y' \in B$, we obtain
\begin{align*}
     \inf_{y, y' \in B} f(tx+(1-t)x', ty+(1-t)y') \le t \inf_{y\in B} f(x,y) + (1-t) \inf_{y' \in B} f(x',y').
\end{align*}
Because $B$ is convex, $\{ty+(1-t)y'|y,y' \in B\} \subset B$, so $\inf_{y, y' \in B} f(tx+(1-t)x', ty+(1-t)y') \ge \inf_{y\in B} f(tx+(1-t)x', y)$ and we get the convexity of $x \in A \mapsto \inf_{y\in B} f(x,y)$.
\end{proof}

\begin{lemma}\label{lemma:convexity_moreau_envelop}
If $g$ is convex, then for all $\gamma > 0$, $g^{\gamma}$ is convex.
\end{lemma}
Note that a result on the preservation of the essential strict convexity, i.e. strict convexity on every convex subset of the domain, has been proved in~\cite[Theorem 3.7]{planiden2019proximal}.
\begin{proof}
We denote $v(x, y) = \frac{1}{2\gamma} \|x - y\|^2 + g(y)$. 
It is clear that $v$ is a sum of two functions that are both convex w.r.t. $(x,y)$.
Then by applying Lemma~\ref{lemma:inf_convex_functions}, we get that $g^{\gamma}$ is convex.
\end{proof}

In order to study the weak convexity of the Moreau envelope, we will need the following technical lemma.
\begin{lemma}\label{lemma:technical_result_weakly_convex_moreau_envelop}
If $g$ is $\rho$-weakly convex with $\gamma\rho<1$, and denoting $g_{\rho} = g + \frac{\rho}{2}\|\cdot\|^2$ which is convex, we have
\begin{align*}
    g^\gamma(x) = g_\rho^{\frac{\gamma}{1-\gamma\rho}}\left(\frac{x}{1-\rho\gamma}\right)-\frac\rho{2(1-\rho \gamma)}\|x\|^2,
\end{align*}
and
\begin{align*}
    \Prox{\gamma g}{x} = \Prox{\frac{\gamma}{1-\gamma \rho} g_{\rho}}{\frac{x}{1-\gamma \rho}}.
\end{align*}
\end{lemma}
\begin{proof}
Let us denote as $g_\rho$ the convex function $g_\rho=g+\frac{\rho}2\|\cdot\|^2$. From the definition of the Moreau envelop~\eqref{eq:moreau_env}, one has $g^\gamma=(g_{\rho}-\frac{\rho}2\|\cdot\|^2)\square \frac{1}{2\gamma}\|\cdot\|^2$, so
\begin{align}
    g^\gamma(x)&=\inf_y g_\rho(y)-\frac{\rho}2\|y\|^2 + \frac{1}{2\gamma}\|x-y\|^2\label{eq:prox1}\\
    &=\inf_y\left(g_\rho(y) +\frac{1-\rho\gamma}{2\gamma}\left(\|y\|^2-\frac2{1-\rho\gamma}\langle x,y\rangle +\frac1{(1-\rho \gamma)^2} \|x\|^2\right)  \right) \nonumber\\
    &+\frac1{2\gamma}\left(1-\frac1{(1-\rho \gamma)}\right)\|x\|^2\nonumber\\
    &=\inf_y\left({g_\rho(y) +\frac{1-\rho\gamma}{2\gamma}\left|\left|y-\frac{1}{1-\rho \gamma} x\right|\right|^2} \right) -\frac\rho{2(1-\rho \gamma)}\|x\|^2\label{eq:prox2}\\
    &=g_\rho^{\frac{\gamma}{1-\gamma\rho}}\left(\frac{x}{1-\rho\gamma}\right)-\frac\rho{2(1-\rho \gamma)}\|x\|^2.\nonumber
\end{align}
The minimum point $y \in \R^d$ that solves problems~\eqref{eq:prox1} and~\eqref{eq:prox2} can respectively be written as $\Prox{\gamma g}{x}$ in~\eqref{eq:prox1} (see relation~\eqref{eq:moreau_prox}) and $\Prox{\frac{\gamma}{1-\gamma \rho} g_{\rho}}{\frac{x}{1-\gamma \rho}}$ in~\eqref{eq:prox2}, using the definition of the proximal operator. Thus we get 
\begin{align*}
    \Prox{\gamma g}{x} = \Prox{\frac{\gamma}{1-\gamma \rho} g_{\rho}}{\frac{x}{1-\gamma \rho}}.
\end{align*}
\end{proof}

\begin{lemma}\label{lemma:moreau_envelop_weakly_convex}
If $g$ is $\rho$-weakly convex, with $\gamma \rho< 1$, then $g^{\gamma}$ is $\frac{\rho}{1 - \gamma \rho}$-weakly convex.
\end{lemma}
Lemma~\ref{lemma:moreau_envelop_weakly_convex} generalizes Lemma~\ref{lemma:convexity_moreau_envelop}, for $\rho>0$. Note that the constant of weak convexity of the Moreau envelope of $g$ is larger than that of $g$ because $\frac{\rho}{1-\gamma\rho} > \rho$.
\begin{proof}

Let us denote by $g_\rho$ the convex function $g_\rho=g+\frac{\rho}2||.||^2$. By Lemma~\ref{lemma:technical_result_weakly_convex_moreau_envelop}, we get
\begin{align*}
    g^\gamma(x)=g_\rho^{\frac{\gamma}{1-\gamma\rho}}\left(\frac{x}{1-\rho\gamma}\right)-\frac\rho{2(1-\rho \gamma)}||x||^2
\end{align*}
Since $g_\rho$ is convex, so does $g_\rho^{\frac{\gamma}{1-\gamma\rho}}\left(\frac{x}{1-\rho\gamma}\right)$ for $\gamma\rho<1$. Hence we get that $g^\gamma + \frac{\rho}{2(1-\gamma\rho)} \|\cdot\|^2$ is convex, so $g^{\gamma}$ is $\frac{\rho}{1-\gamma\rho}$-weakly convex for $\gamma\rho<1$.
\end{proof}
Example~\ref{ex1} provides a direct illustration of this result.
Indeed, for $g(x) = -\frac{\rho}{2} \|x\|^2$, $\rho>0$ and $\gamma \rho < 1$, the Moreau envelop $g^{\gamma}(x) = -\frac{\rho}{2(1 - \gamma \rho)}\|x\|^2$ is $\frac{\rho}{1-\gamma\rho}$-weakly convex.

As shown in the following lemma,  strong convexity is also preserved by the Moreau envelope.\\

\begin{lemma}\label{lemma:strong_convexity_moreau_envelop}
For a function $g$ $\rho$-strongly convex, then $g^{\gamma}$ is $\frac{\rho}{1+\gamma \rho}$ strongly convex.
\end{lemma}
This lemma has been proved in~\cite[Lemma 2.23]{planiden2016strongly}. Notice that the strong convexity of $g^\gamma$ is reduced with respect to that of $g$, as $\frac{\rho}{1+\gamma \rho} < \rho$. This phenomenon is illustrated by Example~\ref{ex1}.
\begin{proof}
Assume that $g$ is $\rho$-strongly convex, then, thanks to Lemma~\ref{lemma:fondamental_properties_convex_conjuguate}(ii) and Lemma~\ref{lemma:fondamental_properties_convex_conjuguate}(iv), $(g^\gamma)^*=(g\Box\frac1{2\gamma}||x||^2)^*=g^*+\frac{\gamma}2||x||^2$  is convex and $\frac1\rho+\gamma=\frac{1+\rho\gamma}{\rho}$-smooth thanks to Lemma~\ref{lemma:more_convex_conjuguate_properties}(iv). Combining Lemma~\ref{lemma:more_convex_conjuguate_properties}(iv) and Lemma~\ref{lemma:fondamental_properties_convex_conjuguate}(iii), we deduce that $(g^\gamma)^{\star \star} = g^\gamma$ is $\frac{\rho}{1+\rho\gamma}$ strongly convex.
\end{proof}

Another way of measuring the convexity of the Moreau envelope has been proposed in~\cite[Lemma 1]{habring2025diffusion} with the non-convexity criteria (NC) of a function defined for $g : \R^d \mapsto \R$ by
\begin{align}\label{eq:nc_criteria}
    \text{NC}(g) = \sup_{x, y \in \R^d, \lambda \in [0, 1]} g(\lambda x + (1 - \lambda y)) - \lambda g(x) - (1 - \lambda) g(y).
\end{align}
By definition, $\text{NC}(g) \le 0$ if and only if $g$ is convex. 
\begin{lemma}\label{lemma:measure_of_convexity}
For $g$ a $\rho$-weakly convex function with $\gamma \rho < 1$, we have $$\text{NC}(g^\gamma) \le \text{NC}(g)$$.
\end{lemma}
Note that Lemma~\ref{lemma:measure_of_convexity} suggests that the Moreau envelope convexifies the  original function, which balances results such as Lemma~\ref{lemma:moreau_envelop_weakly_convex} or Lemma~\ref{lemma:strong_convexity_moreau_envelop}, which show that the (weak and strong) convex constants may be degraded by the Moreau envelope.  Lemma~\ref{lemma:measure_of_convexity} provides another way to prove Lemma~\ref{lemma:convexity_moreau_envelop} as $g$ convex, \textit{i.e.} $\text{NC}(g) \le 0$ implies that $\text{NC}(g^\gamma) \le 0$, \textit{i.e.} $g^\gamma$ is convex.
\begin{proof}
For $x, y \in \R^d$, $p = \Prox{\gamma g}{x}, q = \Prox{\gamma g}{y}$, $x_{\lambda} = \lambda x + (1-\lambda)y$, for $\lambda \in [0,1]$, and $p_{\lambda} = \lambda p + (1-\lambda)q$. From the Moreau envelope definition and the convexity of the $\|\cdot\|^2$ function, we have
\begin{align*}
    &g^\gamma(x_{\lambda}) - \lambda g^\gamma(x) - (1 - \lambda)g^\gamma(y) \\
    &\le g(p_{\lambda}) + \frac{1}{2\gamma}\|p_{\lambda} - x_\lambda\|^2 - \lambda g(p) - \frac{\lambda}{2\gamma}\|p - x\|^2  - (1 - \lambda) g(q) \\
    &- \frac{(1 - \lambda)}{2\gamma}\|q - y\|^2 \\
    &\le g(p_{\lambda})- \lambda g(p)- (1 - \lambda) g(q).
\end{align*}
Then we obtain
\begin{align*}
    \text{NC}(g^\lambda)&=\sup_{x, y \in \R^d, \lambda \in [0, 1]} g^\gamma(x_{\lambda}) - \lambda g^\gamma(x) - (1 - \lambda)g^\gamma(y) \\
    &\le \sup_{p, q \in \R^d, \lambda \in [0, 1]} g(p_{\lambda})- \lambda g(p)- (1 - \lambda)g(q)\\
    &= \text{NC}(g).
\end{align*}
\end{proof}

\section{Properties of the proximal operator for weakly convex functions}\label{sec:properties_prox_operator}
Based on the previous link between the convex conjugate and the Moreau envelope, we can now deduce some properties on the proximal operator (Proposition~\ref{prop:notation_prox_operator} and Lemma~\ref{lemma:prox_lipschitz_weakly_cvx}) of weakly convex functions. Moreover, we will study how the Moreau envelope preserves strong convexity (Lemma~\ref{lemma:strong_convexity_moreau_envelop}).

\begin{proposition}\label{prop:notation_prox_operator}
For $g$ $\rho$-weakly convex with $\gamma \rho < 1$ and differentiable at $x\in \R^d$, we have
\begin{align}\label{eq:inverse_prox}
    \Prox{\gamma g}{x + \gamma \nabla g(x)} = x.
\end{align}
Moreover, if $g$ is differentiable on $\R^d$, then $\mathsf{Prox}_{\gamma g}$ is bijective, $\mathsf{Prox}_{\gamma g}^{-1} = I_d + \gamma \nabla g$ and $\mathsf{Prox}_{\gamma g}(\R^d)$ is homeomorphic with $\R^d$.
\end{proposition}
\begin{proof}
We first demonstrate relation~\eqref{eq:inverse_prox} for convex functions $g$ and then use it for weakly convex ones.
\begin{itemize}
    \item 
For $g$ convex and $x \in \R^d$, by Lemma~\ref{lemma:more_convex_conjuguate_properties}(ii), we have for $y = \nabla g(x)$, $g(x) + g^\star(y) = \langle x, y\rangle$. Then by Proposition~\ref{prop:dual_result_prox}, we get $x = \Prox{g}{x + \nabla g(x)}$. Then, if $g$ is convex, $\gamma g$ is convex and we get $\forall \gamma > 0$, 
\begin{equation}\label{rel:g_conv}
    \Prox{\gamma g}{x + \gamma \nabla g(x)} = x.
\end{equation}
 \item 
If $g$ is $\rho$-weakly convex with $\gamma \rho < 1$, then we introduce the convex function  $g_{\rho} = g + \frac{\rho}{2}\|\cdot\|^2$. By Lemma~\ref{lemma:technical_result_weakly_convex_moreau_envelop}, we have
\begin{align*}
    \Prox{\gamma g}{x} = \Prox{\frac{\gamma}{1-\gamma \rho} g_{\rho}}{\frac{x}{1-\gamma \rho}}.
\end{align*}
Then 
\begin{align*}
    \Prox{\gamma g}{x + \gamma \nabla g(x)} &= \Prox{\frac{\gamma}{1-\gamma \rho} g_{\rho}}{\frac{x + \gamma \nabla g(x)}{1-\gamma \rho}} \\
    &= \Prox{\frac{\gamma}{1-\gamma \rho} g_{\rho}}{\frac{x + \gamma \nabla g_{\rho}(x) - \gamma \rho x}{1-\gamma \rho}} \\
    &= \Prox{\frac{\gamma}{1-\gamma \rho} g_{\rho}}{x + \frac{\gamma}{1-\gamma \rho} \nabla g_{\rho}(x)} \\
    &= x.
\end{align*}
The last equality is obtained by applying~\eqref{rel:g_conv} on the convex function $g_{\rho}$.
\end{itemize}

Now assuming $g$ differentiable on $\R^d$, we have $\forall x \in \R^d$ 
\begin{align*}
    \Prox{\gamma g}{x + \gamma \nabla g(x)} = x.
\end{align*}
Since $g$ is $\rho$-weakly convex, $I_d + \gamma g$ is $1 - \gamma \rho$ strongly convex. Thus, its gradient $I_d + \gamma \nabla g$ is bijective. We obtain that $\mathsf{Prox}_{\gamma g}$ is bijective and its inverse is $I_d + \gamma \nabla g$.

By Proposition~\ref{prop:derivative_in_gamma}, $x \in \R^d \mapsto \Prox{\gamma g}{x}$ is continuous. Moreover, it is bijective, so it is a homeomorphism. Therefore, its image $\Prox{\gamma g}{\R^d}$ is homeomorphic to $\R^d$ and, in particular, it is an open set.
\end{proof}

\begin{lemma}\label{lemma:expansivity_prox}
    For $g$ $\rho$-weakly convex with $\gamma \rho < 1$ and $x, y \in \R^d$, we have
    \begin{align}
        \|\Prox{\gamma g}{x} - \Prox{\gamma g}{y}\|^2 \le \frac{1}{1 - \gamma \rho} \langle x - y, \Prox{\gamma g}{x} - \Prox{\gamma g}{y} \rangle
    \end{align}
\end{lemma}
Lemma~\ref{lemma:expansivity_prox} shows that $\mathsf{Prox}_{\gamma g}$ is $(1 - \gamma \rho)$-cocoercive which implies its Lipschitzness and and its monotonic characters. It has been proved for instance in~\cite[Proposition 3.3]{hoheiselproximal}.

\begin{proof}
For $x, y \in \R^d$, $p = \Prox{\gamma g}{x}$ and $q = \Prox{\gamma g}{y}$, by Lemma~\ref{lemma:technical_lemma_prox} we have
\begin{align*}
    g(p) +\frac{1}{\gamma} \langle x - p, q - p \rangle \le g(q) + \frac{\rho}{2} \|q - p\|^2 \\
    g(q) +\frac{1}{\gamma} \langle y - q, p - q \rangle \le g(p) + \frac{\rho}{2} \|q - p\|^2.
\end{align*}
Summing the two previous inequalities gives
\begin{equation*}
\begin{split}
    \frac{1}{\gamma} \langle q - p + x - y, q - p \rangle &\le \rho \|q - p\|^2 \\
    \left(1 - \gamma\rho \right) \|q - p\|^2 &\le \langle x - y, p - q \rangle. 
\end{split}
\end{equation*}\
\end{proof}

\begin{lemma}\label{lemma:prox_lipschitz_weakly_cvx}
For $g$ $\rho$-weakly convex with $\gamma \rho < 1$, $\mathsf{Prox}_{\gamma g}$ is $\frac{1}{1-\gamma \rho}$-Lipschitz.
\end{lemma}
Note that in particular, if $g$ is convex, then $\rho=0$ and $\forall \gamma > 0$, $\mathsf{Prox}_{\gamma g}$ is $1$-Lipschitz. The fact that the proximal operator is Lipschitz is important for guaranteeing its stability. The proof of this Lemma~\ref{lemma:prox_lipschitz_weakly_cvx} can be found in~\cite[Proposition 2]{gribonval2020characterization} where the authors used~\cite[Proposition 5.b]{moreau1965proximite}.
\begin{proof}
By applying the Cauchy-Schwarz inequality in Lemma~\ref{lemma:expansivity_prox}, we directly obtain Lemma~\ref{lemma:prox_lipschitz_weakly_cvx}.
\end{proof}

\section{The Moreau envelope is differentiable and preserves the minimizers and the critical points}\label{sec:moreau_diff_optim}
In this section, we prove (Proposition~\ref{prop:nabla_g_weakly_cvx_appendix}) that the Moreau envelope is differentiable even if $g$ is not. This result also allows to interpret the proximal operator as a gradient-step on the Moreau envelope.  Then, we show that the Moreau envelope has Lipschitz gradient (Lemma~\ref{lemma:moreau_envelop_smooth}),  which is an important regularity property to ensure the stability of optimization algorithms. Finally, we prove that the Moreau envelope preserves the  minimal points (Proposition~\ref{prop:same_argmin}) and the critical points (Lemma~\ref{lemma:same_critical_points}) of the original function. 

\begin{proposition}\label{prop:nabla_g_weakly_cvx_appendix}
For a $\rho$-weakly convex function $g$ and for $\gamma \rho < 1$, we have
\begin{align*}
    \nabla g^{\gamma}(x) = \frac{1}{\gamma} \left( x - \Prox{\gamma g}{x} \right).
\end{align*}
\end{proposition}
Combining Proposition~\ref{prop:nabla_g_weakly_cvx_appendix}  with Proposition~\ref{prop:derivative_in_gamma} proves that the Moreau envelope solves an Hamilton-Jacobi equation~\cite[Theorem 5, Section 3.3.2]{evans2022partial} for $g$ $\rho$-weakly convex, i.e.
\begin{align}
    \frac{\partial g^{\gamma}}{\partial \gamma}(x) + \frac{1}{2} \|\nabla g^{\gamma}(x)\|^2 = 0.
\end{align}
The fact that the Moreau envelope satisfies a Hamilton-Jacobi equation has motivated its use for regularization in Hilbert spaces~\cite{lasry1986remark, daniilidis2018explicit} to solve, for instance, viscosity equations~\cite{crandall1983viscosity}.
\begin{proof}
Our proof is based on the strategy detailed in~\cite[Theorem 2.6]{rockafellar2009variational} for $g$ convex, which we generalize for $g$ $\rho$-weakly convex. This lemma is formulated in the weakly convex setting in~\cite[Lemma 2.2]{boct2023alternating} and in~\cite[Lemma 2.2]{davis2019stochastic} without proofs. A proof can be found in~\cite[Lemma 3]{habring2025diffusion}.

For $x \in \R^d$, we introduce $v = \Prox{\gamma g}{x}$, $w = \frac{1}{\gamma} \left( x - v \right)$ and $h(u) = g^{\gamma}(x + u) - g^{\gamma}(x) - \langle w, u \rangle$. We aim at proving that $g^{\gamma}$ is differentiable at $x$ with $\nabla g^{\gamma}(x) = w$, which is equivalent to $h$ differentiable at $0$ and $\nabla h(0) = 0$.

By definition of the Moreau envelope, since $v = \Prox{\gamma g}{x}$, we have $g^{\gamma}(x) = \frac{1}{2\gamma} \|x - v\|^2 + g(v)$ and $g^{\gamma}(x+u) \le \frac{1}{2\gamma} \|x + u - v\|^2 + g(v)$. So, we get
\begin{align}
    h(u) &\le \frac{1}{2\gamma} \|x + u - v\|^2 - \frac{1}{2\gamma} \|x - v\|^2 - \langle w, u \rangle \nonumber\\
    &\le \frac{1}{2\gamma} \|u\|^2 + \frac{1}{\gamma} \langle x - v, u \rangle - \langle w, u \rangle \nonumber\\
    &\le \frac{1}{2\gamma} \|u\|^2.\label{eq:tmp}
\end{align}

Moreover, as $g^{\gamma}(x + u)$ is $\frac{\rho}{1 - \rho \gamma}$-weakly convex (Lemma~\ref{lemma:moreau_envelop_weakly_convex}), we have that    $h(u) + \frac{\rho}{2(1 - \rho \gamma)}\|u\|^2 = g^{\gamma}(x + u) + \frac{\rho}{2(1 - \rho \gamma)}\|u\|^2 - g^{\gamma}(x) - \langle w, u \rangle$ is convex in~$u$. 
We thus obtain
\begin{align*}
    \frac{1}{2}\left(h(u) + \frac{\rho}{2(1 - \rho \gamma)} \|u\|^2 + h(-u) + \frac{\rho}{2(1 - \rho \gamma)} \|u\|^2 \right) \ge h(0) = 0 \\
    h(u) \ge -\frac{\rho}{(1 - \rho \gamma)} \|u\|^2 - h(-u) \ge \left( - \frac{\rho}{1 - \rho \gamma} - \frac{1}{2\gamma} \right) \|u\|^2
\end{align*}
which leads to
\begin{align*}
     \left( - \frac{\rho}{1 - \rho \gamma} - \frac{1}{2\gamma} \right) \|u\|^2 \le h(u) \le \frac{1}{2\gamma}\|u\|^2.
\end{align*}
Therefore $h$ is differentiable at $0$ and $\nabla h(0) = 0$. This concludes the proof.
\end{proof}

We can deduce that the Moreau envelope is smooth, \textit{i.e.} $\nabla g^{\gamma}$ is Lipschitz.

\begin{lemma}\label{lemma:moreau_weakly_cvx_env_smooth}
If $g$ is $\rho$-weakly convex, with $\gamma \rho < 1$, then $g^{\gamma}$ is $L$ smooth, \textit{i.e.} $\nabla g^\gamma$ is $L$-Lipschitz, with
\begin{align*}
    L = \left\{
    \begin{array}{ll}
        \frac{1}{\gamma} & \mbox{if } \gamma \rho \le \frac{1}{2} \\
        \frac{\rho}{1 - \gamma \rho} & \mbox{if } \gamma \rho \ge \frac{1}{2}.
    \end{array}
\right.
\end{align*}
\end{lemma}
Hence $\nabla g^\gamma$ is $\max\left(\frac{1}{\gamma}, \frac{\rho}{1-\gamma\rho} \right)$-Lipschitz as noted in~\cite[Lemma 3.1]{bohm2021variable}. In the literature, a sub-optimal Lipschitz constant for $\nabla g^\gamma$ has been used $\frac{2 - \rho \gamma}{\gamma(1 - \rho \gamma)} > \max\left(\frac{1}{\gamma}, \frac{\rho}{1-\gamma\rho} \right)$, see for instance~\cite[Corollary 1]{habring2025diffusion} or~\cite[Proposition E.C.2]{sun2022algorithmsdifferenceofconvexdcprograms}. To the best of our knowledge, the Lipschitz constant defined in Lemma~\ref{lemma:moreau_weakly_cvx_env_smooth} is the sharpest known in the literature.\\

\begin{proof}
We rely on the proof provided in~\cite[Corollary 3.4]{hoheiselproximal}.
For $x, y \in \R^d$, $p = \Prox{\gamma g}{x}$ and $q = \Prox{\gamma g}{y}$, by applying Lemma~\ref{lemma:expansivity_prox}, we have
\begin{align*}
    &\|(x - p)  - (y - q)\|^2 = \|x - y\|^2 + \|p - q\|^2 + 2 \langle x - y, q - p \rangle \\
    &= \|x - y\|^2 + \left( \frac{1}{1 -\gamma \rho} - 2 \right) \langle x - y, p - q \rangle + \|p - q\|^2 - \frac{1}{1 -\gamma \rho} \langle x - y, p-q \rangle \\
    &\le \|x - y\|^2 + \left( \frac{1}{1 -\gamma \rho} - 2 \right) \langle x - y, p - q \rangle.
\end{align*}

For $\gamma \rho \le \frac{1}{2}$, we have $\frac{1}{1 -\gamma \rho} - 2 \le 0$. Due to Lemma~\ref{lemma:expansivity_prox}, $\langle x - y, p - q \rangle \ge 0$, so we get
\begin{align*}
    \|(x - p)  - (y - q)\|^2 \le \|x - y\|^2.
\end{align*}
So, by Proposition~\ref{prop:nabla_g_weakly_cvx_appendix} for $\gamma \rho \le \frac{1}{2}$, we have
\begin{align*}
    \|\nabla g^\gamma(x)  - \nabla g^\gamma(x)\| \le \frac{1}{\gamma} \|x - y\|.
\end{align*}

For $\gamma \rho \ge \frac{1}{2}$, thanks to the Cauchy-Schwarz inequality and Lemma~\ref{lemma:prox_lipschitz_weakly_cvx}, we have
\begin{align*}
    \|(x - p)  - (y - q)\|^2 &\le \|x - y\|^2 + \left( \frac{1}{1 -\gamma \rho} - 2 \right) \| x - y\| \|p - q\| \\
    &\le \|x - y\|^2 + \left( \frac{1}{1 -\gamma \rho} - 2 \right) \frac{1}{1 -\gamma \rho} \| x - y\|^2 \\
    &= \left(\frac{\gamma \rho}{1 -\gamma \rho}\right)^2 \| x - y\|^2.
\end{align*}
Then, by applying Proposition~\ref{prop:nabla_g_weakly_cvx_appendix} for $\gamma \rho \ge \frac{1}{2}$, we have
\begin{align*}
    \|\nabla g^\gamma(x)  - \nabla g^\gamma(x)\| \le \frac{ \rho}{1 -\gamma \rho} \|x - y\|.
\end{align*}
\end{proof}

\begin{proposition}\label{prop:same_argmin}
For $g$ a $\rho$-weakly convex function with $\gamma \rho < 1$, the following properties are equivalent
\begin{enumerate}
  \renewcommand{\labelenumi}{\roman{enumi})}
  \item $x \in \argmin g$
  \item $x \in \argmin g^\gamma$
\end{enumerate}
\end{proposition}
Proposition~\ref{prop:same_argmin} is proved in~\cite[Corollary 3.4]{hoheiselproximal} using convex optimization results. Here, we provide an elementary proof.
\begin{proof}
For the first implication, if $x \in \argmin g$, we have that
\begin{align*}
    g^{\gamma}(x) &= \inf_{y \in \R^d} \frac{1}{2\gamma}\|x - y\|^2 + g(y) \\
    &= \inf_{y \in \R^d} \frac{1}{2\gamma}\| y\|^2 + g(x - y)\\
    &\ge \inf_{y \in \R^d} \frac{1}{2\gamma}\| y\|^2 + g(x) = g(x).
\end{align*}
So we get that $g^\gamma(x) = g(x)$. Then for $z \in \R^d$, we have
\begin{align*}
    g^\gamma(z) = \inf_{y \in \R^d} g(z-y) +\frac{1}{2\gamma}\|y\|^2 \ge g(z) \ge g(x) = g^\gamma(x).
\end{align*}
So we get that $x \in \argmin g^\gamma$ and (i) implies (ii).

Conversely, suppose that $x \in \argmin g^\gamma$. From Proposition~\ref{prop:nabla_g_weakly_cvx_appendix}, we have
\begin{align*}
    0 = \nabla g^\gamma(x) = \frac{1}{\gamma} \left(x - \Prox{\gamma g}{x}\right).
\end{align*}
So we obtain that $\Prox{\gamma g}{x} = x$. By definition of the Moreau envelope, 
\begin{align*}
    g^\gamma(x) = g(\Prox{\gamma g}{x}) + \frac{1}{2\gamma}\|x - \Prox{\gamma g}{x}\|^2 = g(x).
\end{align*}

Then, for $y \in \R^d$, by  optimality of $x$ we obtain
\begin{align*}
    g(x) = g^\gamma(x) \le g^\gamma(y) \le g(y),
\end{align*}
so $x \in \argmin g$, which proves that (ii) implies (i).
\end{proof}

\begin{lemma}\label{lemma:caracterization_critical_point_moreau_env}
If $g$ is $\rho$-weakly convex with $\gamma \rho < 1$, the following properties are equivalent, for $x \in \R^d$,
\begin{enumerate}
  \renewcommand{\labelenumi}{\roman{enumi})}
  \item $\nabla g^\gamma(x) = 0$
  \item $\Prox{\gamma g}{x} = x$
  \item $g^\gamma(x) = g(x)$
  \item $g(\Prox{\gamma g}{x}) = g(x)$
\end{enumerate}
\end{lemma}
Lemma~\ref{lemma:caracterization_critical_point_moreau_env} gives simple characterizations of the critical points of $g^\gamma$. In the case of convex $g$, the criterion given in Lemma~\ref{lemma:caracterization_critical_point_moreau_env} define the global minimum of $g$ and $g^\gamma$. This is due to the fact that a critical point of a convex function is a global minimum.
\begin{proof}
By Proposition~\ref{prop:nabla_g_weakly_cvx_appendix},  (i) is equivalent to (ii).
If $\Prox{\gamma g}{x} = x$, then
\begin{align*}
    g^\gamma(x) = g(\Prox{\gamma g}{x}) + \frac{1}{2\gamma}\|x - \Prox{\gamma g}{x}\|^2 = g(x),
\end{align*}
which proves that (ii) implies (iii).

If $g^\gamma(x) = g(x)$, then by Proposition~\ref{prop:moreau_envelop_dependence_in_gamma}, $\forall \lambda \in (0, \gamma]$, $g^\lambda(x) = g(x)$. So we get that $\frac{\partial g^\gamma}{\partial \gamma}(x) = 0$ Then by Proposition~\ref{prop:derivative_in_gamma}, we get that $\Prox{\gamma g}{x} = x$, then $g(\Prox{\gamma g}{x}) = g(x)$, which proves that (iii) implies (iv).

If $g(\Prox{\gamma g}{x}) = g(x)$, by definition of the proximal operator, we have
\begin{align*}
    g(\Prox{\gamma g}{x}) + \frac{1}{2\gamma}\|x - \Prox{\gamma g}{x}\|^2 &\le g(x) \\
    g(x) + \frac{1}{2\gamma}\|x - \Prox{\gamma g}{x}\|^2 &\le g(x) \\
    \Prox{\gamma g}{x} &= x.
\end{align*}
This proves that (iv) implies (ii) and finishes the proof.
\end{proof}

\begin{lemma}\label{lemma:same_critical_points}
For $g$ a $\rho$-weakly convex function with $\gamma \rho < 1$, if $g$ is differentiable at $x \in \R^d$, then  $\nabla g(x) = 0$ if and only if $\nabla g^\gamma(x) = 0$.
\end{lemma}
The previous lemma shows that the critical points of $g$ and $g^\gamma$ are the same. Note that this lemma can be generalized tor non-differentiable functions $g$ using  Clarke's sub-differentials, see~\cite[Corollary 3.4]{hoheiselproximal}. We observe in example~\ref{ex:non_convex_function} that the critical points and the minimizers are preserved by the Moreau envelope.
\begin{proof}
If $\nabla g(x) = 0$, by Proposition~\ref{prop:notation_prox_operator}, we have $\Prox{\gamma g}{x} = x$. Then by Proposition~\ref{prop:nabla_g_weakly_cvx_appendix}, $\nabla g^\gamma(x) = 0$.

Next, if $\nabla g^\gamma(x) = 0$, by Proposition~\ref{prop:nabla_g_weakly_cvx_appendix}, we have $\Prox{\gamma g}{x} = x$. Then by the optimal condition of the proximal operator, we have
\begin{align*}
    \nabla g(\Prox{\gamma g}{x}) + \frac{1}{\gamma} \left( \Prox{\gamma g}{x} -x \right) = 0
\end{align*}
and we obtain $\nabla g(x) = 0$.
\end{proof}

\section{Second derivative of the Moreau envelope and convexity of the image of the proximal operator}\label{sec:second_derivation_mro_env_imgae_prox}

For weakly convex functions, we now prove that the Moreau envelope is twice differentiable if $g$ is twice differentiable and Lipschitz on the image of the proximal operator. These second-order properties have been studied in previous works~\cite{qi1994second, lucet1995formule, lemarechal1997practical} in the convex setting. 
We will also show that the boundary of the  convex envelope of $\Prox{\gamma g}{\R^d}$ is of measure zero (Proposition~\ref{prop:leb_measure_prox_image_estimation}). In other words, the image of the proximity operator of a weakly convex function is almost convex.


\begin{lemma}\label{lemma:nabla_2_g_weakly_cvx}
Let $g$ be a $\rho$-weakly convex function. Assume that  $g$ that is  $\mathcal{C}^2$  and $L_g$-smooth on $\Prox{\gamma g}{\R^d}$, \textit{i.e.} $-L_g I_d \preceq \nabla^2 g \preceq L_g I_d$, with $\rho \gamma < 1$ and $L_g \gamma < 1$. Then we get, for $x \in \R^d$
\begin{align*}
    \mathsf{J}_{\mathsf{Prox}_{\gamma g}}(x) &= \left( I_d + \gamma \nabla^2g(\Prox{\gamma g}{x}) \right)^{-1} \\
    \nabla^2 g^{\gamma}(x) &= \frac{1}{\gamma}\left( I_d - \left( I_d + \gamma \nabla^2g(\Prox{\gamma g}{x}) \right)^{-1} \right),
\end{align*}
where $\mathsf{J}_{\mathsf{Prox}_{\gamma g}}$ states for the Jacobian matrix of $\mathsf{Prox}_{\gamma g}$.
\end{lemma}
By Proposition~\ref{prop:notation_prox_operator}, $\Prox{\gamma g}{\R^d}$ is an open set, so the Lipschitness and $\mathcal{C}^2$ properties of $g$ are well defined on this open set.  Notice that 
Lemma~\ref{lemma:nabla_2_g_weakly_cvx} has been shown for convex functions  in~\cite[Theorem 3.4]{ovcharova2010second}.
\begin{proof}
If $g$ is twice differentiable, by Proposition~\ref{prop:notation_prox_operator}, we get that
\begin{align*}
    \Prox{\gamma g}{x} = \left(I_d + \gamma \nabla g \right)^{-1}(x).
\end{align*}
The previous equation proves that $\mathsf{Prox}_{\gamma g}$ is differentiable if $g$ is twice differentiable. By differentiating  $\Prox{\gamma g}{x} + \gamma \nabla g(\Prox{\gamma g}{x})$, we obtain
\begin{align*}
    \mathsf{J}_{\mathsf{Prox}_{\gamma g}} = \left(  I_d + \gamma \nabla^2 g \circ \mathsf{Prox}_{\gamma g} \right)^{-1}.
\end{align*}
Next, by the optimal condition of the proximal operator, we get 
\begin{align*}
    \frac{1}{\gamma} \left( \Prox{\gamma g}{x} - x \right) + \nabla g(\Prox{\gamma g}{x}) = 0.
\end{align*}
Combined with Proposition~\ref{prop:nabla_g_weakly_cvx_appendix}, we obtain
\begin{align}\label{eq:optimal_condition_on_gradient}
    \nabla g^{\gamma}(x) = \nabla g(\Prox{\gamma g}{x}).
\end{align}
Then, we differentiate Equation~\eqref{eq:optimal_condition_on_gradient} to get
\begin{align*}
    \nabla^2 g^{\gamma}(x) &= \nabla^2 g(\Prox{\gamma g}{x}) \mathsf{J}_{\mathsf{Prox}_{\gamma g}}(x) \\
    &=\nabla^2 g(\Prox{\gamma g}{x}) \left(  I_d + \gamma \nabla^2 g(\Prox{\gamma g}{x}) \right)^{-1} \\
    &= \frac{1}{\gamma}\left( I_d - \left( I_d + \gamma \nabla^2g(\Prox{\gamma g}{x}) \right)^{-1} \right),
\end{align*}
which shows the second part of Lemma~\ref{lemma:nabla_2_g_weakly_cvx}.
\end{proof}

We can now deduce that for $\gamma$ sufficiently small, the Moreau envelope is smooth (with a constant independent of $\gamma$) on all $\R^d$ if the function $g$ is smooth on $\Prox{\gamma g}{\R^d}$.

\begin{lemma}\label{lemma:moreau_envelop_smooth}
Let $g$ be a $\rho$-weakly convex function. Assume that $g$ is $\mathcal{C}^2$ and $L_g$-smooth on $\Prox{\gamma g}{\R^d}$ with $L_g \gamma \le \frac{1}{2}$ and $\rho \gamma < 1$. Then $g^{\gamma}$ is $2L_g$-smooth on~$\R^d$.
\end{lemma}
\begin{proof}
By Lemma~\ref{lemma:nabla_2_g_weakly_cvx}, we have
\begin{align*}
    \nabla^2 g^{\gamma}(x) &= \frac{1}{\gamma}\left( I_d - \left( I_d + \gamma \nabla^2g(\Prox{\gamma g}{x}) \right)^{-1} \right).
\end{align*}

As we assume that $-L_g I_d \preceq \nabla^2g(\Prox{\gamma g}{x}) \preceq L_g I_d$, we get
\begin{align*}
    \frac{1}{\gamma}\left( 1 - \frac{1}{ 1 - L_g \gamma} \right) I_d &\preceq \nabla^2 g^{\gamma}(x) \preceq \frac{1}{\gamma}\left( 1 - \frac{1}{ 1 + L_g \gamma} \right) I_d\\
    \underbrace{-\frac{L_g}{ 1 - L_g \gamma}}_{:=u(\gamma)}I_d &\preceq \nabla^2 g^{\gamma}(x) \preceq\underbrace{\frac{L_g}{ 1 + L_g \gamma}}_{:=v(\gamma)}I_d 
\end{align*}

Since $u'(\gamma)=-\frac{L_g^2}{(1 - L_g \gamma)^2} \le 0$, $u$ is non-increasing. As  $u(\frac{1}{2L_g}) = -2L_g$, we obtain that $u(\gamma) \ge -2L_g$,  for all $\gamma \in [0, \frac{1}{2L_g}]$.
In the same way, $v'(\gamma)=-\frac{L_g^2}{(1 + L_g \gamma)^2} \le 0$, so
 $v$ is .  Since $v(0) = L_g$, for $\gamma \ge 0$, we have $v(\gamma) \le L_g$. Finally, for $\gamma \in [0, \frac{1}{2L_g}]$, we get
\begin{align*}
    -2L_g I_d \preceq \nabla^2 g^{\gamma}(x) \preceq L_g I_d.
\end{align*}
So $g^{\gamma}$ is $2L_g$-smooth on $\R^d$.
\end{proof}

\begin{lemma}\label{lemma:high_order_diff}
Let $g$ be a $\rho$-weakly convex function and $\mathcal{C}^k$ in $x \in \R^d$ with $k \in \N^\star$. Then for $\rho \gamma < 1$, $g^\gamma$ is $\mathcal{C}^k$ in $x \in \R^d$.
\end{lemma}
Lemma~\ref{lemma:high_order_diff} is a generalization of~\cite[Theorem 3.12]{planiden2019proximal} in the weakly convex setting in.
\begin{proof}
By Proposition~\ref{prop:notation_prox_operator} and Proposition~\ref{prop:nabla_g_weakly_cvx_appendix}, we have that
\begin{align*}
    \nabla g^\gamma = \frac{1}{\gamma}\left(I_d - \left(I_d + \gamma \nabla g \right)^{-1} \right).
\end{align*}
In the previous equation, the right-term is $\mathcal{C}^{k-1}$. Thus we obtain that $g^\gamma$ is $\mathcal{C}^k$.
\end{proof}

We end this document  by showing that the image of the proximity operator of a weakly convex function is almost convex.
\begin{proposition}\label{prop:leb_measure_prox_image_estimation}
Let $g$ be a $\rho$-weakly convex function with $\rho \gamma < 1$. We have 
\begin{align}\label{eq:leb_measure_dom_image_nulle}
    \text{Leb}(\text{dom}(g) \setminus \Prox{\gamma g}{\R^d}) = 0,
\end{align}
with $\text{Leb}$ the Lebesgue measure and $\text{dom}(g) = \{x \in \R^d\, |\, g(x) < +\infty\}$. Moreover, we have that
\begin{align}\label{eq:leb_measure_conv_image_nulle}
    \text{Leb}(\text{Conv}(\Prox{\gamma g}{\R^d}) \setminus \Prox{\gamma g}{\R^d}) = 0,
\end{align}
where $\text{Conv}(\Prox{\gamma g}{\R^d})$ is the convex envelope of $\Prox{\gamma g}{\R^d}$.
\end{proposition}

Proposition~\ref{prop:leb_measure_prox_image_estimation} has been proved recently in~\cite[Lemma 1(iv)]{renaud2025stabilitylangevindiffusionconvergence}.

\begin{proof}
We denote by $\text{dom}(g)$ the set where $g < + \infty$. By the weak convexity of $g$, we know that $\text{dom}(g)$ is convex. In fact, for $x, y \in \text{dom}(g)$ and $\lambda \in [0,1]$, $g + \frac{\rho}{2}\|\cdot\|^2$ is convex, so we have $g(\lambda x + (1-\lambda)y) \le \lambda g(x) + \frac{\rho\lambda}{2}\|x\|^2 + (1-\lambda)g(y) +\frac{\rho(1-\lambda)}{2}\|x\|^2 - \frac{\rho}{2}\|\lambda x + (1-\lambda)y\|^2 < +\infty$, and $\lambda x + (1-\lambda)y \in \text{dom}(g)$. By~\cite[Theorem 1]{lang1986note}, because $\text{dom}(g)$ is convex, we have that $\text{Leb}(\text{dom}(g) \setminus \text{int}(\text{dom}(g))) = 0$.

By~\cite[Theorem 4.2.3]{hiriart1996convex}, a convex function is differentiable almost everywhere on the interior of its domain. As  $g +\frac{\rho}{2}\|\cdot\|^2$ is convex, it is thus  differentiable almost everywhere on the interior of its domain. As a consequence  $g$ is differentiable almost everywhere on the interior of its domain, $\text{int}(\text{dom}(g))$. Denoting as $\Gamma \subset \text{int}(\text{dom}(g))$ the subset on which $g$ is differentiable, we thus have $\text{Leb}(\text{int}(\text{dom}(g)) \setminus \Gamma) = 0$.

Moreover, for $x \in \Gamma$, $g$ is differentiable at $x$, so by Proposition~\ref{prop:notation_prox_operator}, we have $x = \Prox{\gamma g}{x + \gamma \nabla g(x)} \in \Prox{\gamma g}{\R^d}$. So we get $\Gamma \subset \Prox{\gamma g}{\R^d}$. Therefore, we can deduce that $\text{Leb}(\text{int}(\text{dom}(g)) \setminus \Prox{\gamma g}{\R^d}) = 0$.

Combining the fact that $\text{Leb}(\text{dom}(g) \setminus \text{int}(\text{dom}(g))) = 0$ and $\text{Leb}(\text{int}(\text{dom}(g)) \setminus \Prox{\gamma g}{\R^d}) = 0$, we get that $\text{Leb}(\text{dom}(g) \setminus \Prox{\gamma g}{\R^d}) = 0$, which shows equation~\eqref{eq:leb_measure_dom_image_nulle}.

Moreover for $x \in \R^d$ and $y \in \text{dom}(g)$, by definition of the proximal operator, we get $g(\Prox{\gamma g}{x}) + \frac{1}{2\gamma}\|x - \Prox{\gamma g}{x}\|^2 \le g(y) + \frac{1}{2\gamma}\|x - y\|^2$. So $g(\Prox{\gamma g}{x}) < + \infty$, which gives  $\Prox{\gamma g}{\R^d} \subset \text{dom}(g)$. Because $\text{dom}(g)$ is convex, we get that $\text{Conv}(\Prox{\gamma g}{\R^d}) \subset \text{dom}(g)$. Combined with~\eqref{eq:leb_measure_dom_image_nulle}, we obtain  equation~\eqref{eq:leb_measure_conv_image_nulle}.
\end{proof}

\section{Numerical examples}
In this section, we present various examples of weakly convex functions and their associated Moreau envelope. 

\begin{example}\label{ex:non_convex_function}
We first study the function
\begin{align*}
    h(x) = \left\{
    \begin{array}{ll}
        \frac{1}{4} - \frac{1}{2} x^2 & \mbox{if } |x| \le \frac{1}{2} \\
        \frac{1}{2} (|x| - 1)^2 & \mbox{otherwise.}
    \end{array}
\right.
\end{align*}
$h$ is $\mathcal{C}^1$ and $1$-weakly convex and not convex. After some computation, we obtain that, for $\gamma \in (0, 1)$
\begin{align*}
    h^\gamma(x) &= \left\{
    \begin{array}{ll}
        \frac{1}{4} - \frac{x^2}{2(1 - \gamma)} & \mbox{if } |x| \le \frac{1- \gamma}{2}\\
        \frac{(|x| - 1)^2}{2(1+\gamma)} & \mbox{otherwise.}
    \end{array}
\right. \\
\Prox{\gamma h}{x} &= \left\{
    \begin{array}{ll}
        \frac{x}{1 - \gamma} & \mbox{if } |x| \le \frac{1- \gamma}{2} \\
        \frac{x + \gamma}{1+\gamma} & \mbox{if } x \ge \frac{1- \gamma}{2} \\
        \frac{x - \gamma}{1+\gamma} & \mbox{otherwise.}
    \end{array}
\right.
\end{align*}

On Figure~\ref{fig:an_example}, we present the graph of the function $h$ and $h^\gamma$ for some $\gamma \in (0, 1)$, as well as the corresponding proximal operators. We observe that $h^\gamma$ is weakly convex but not convex. Moreover, the critical points and the global minimum are preserved.

\begin{figure}[!ht]
\centering
\includegraphics[width=\textwidth]{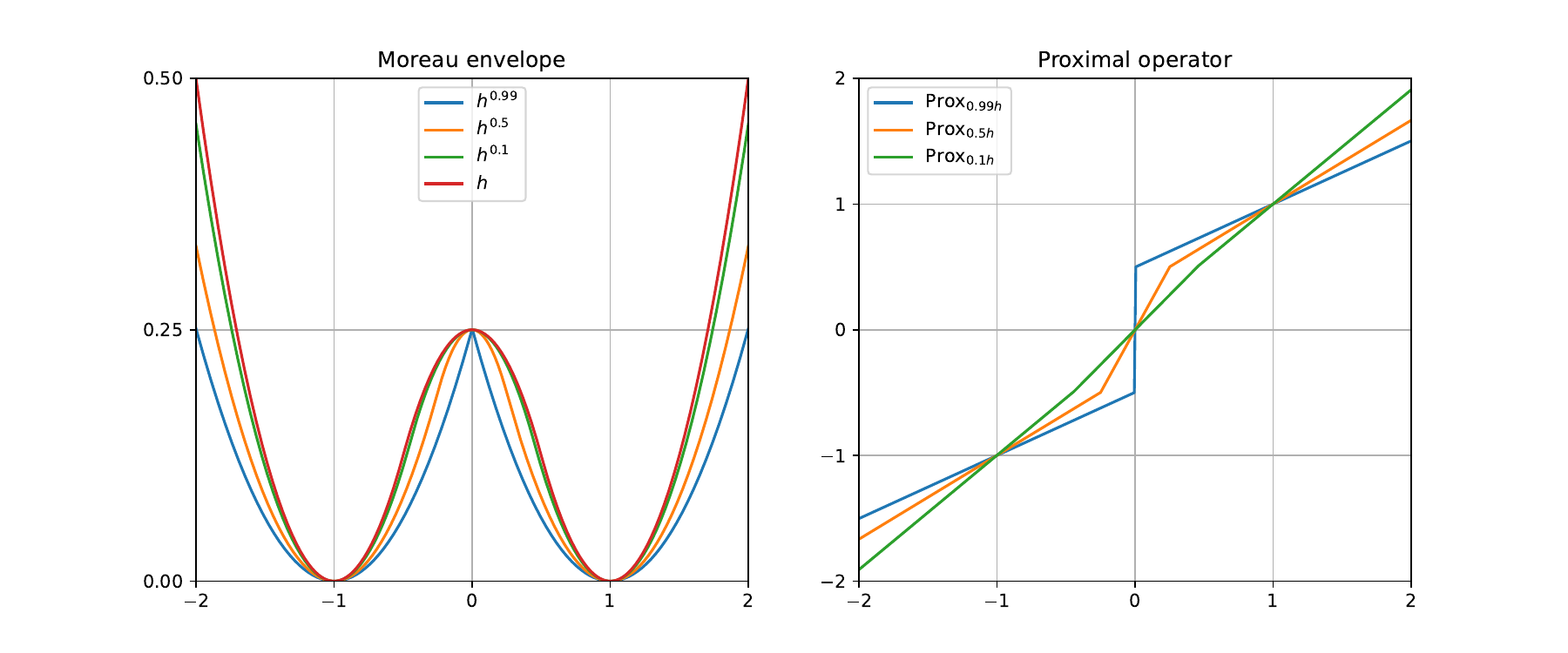}
\caption{Graph of the function $h$, $h^\gamma$ for $\gamma \in \{0.1, 0.5, 0.99\}$ and their associated proximal operators.}
\label{fig:an_example}
\end{figure}
\end{example}

\begin{example}\label{ex:dis_function}
We study the function
\begin{align*}
    f(x) = \left\{
    \begin{array}{ll}
        \frac{1}{2}(x + 1)^2 & \mbox{if } x \le -\frac{1}{2} \\
        \frac{1}{4} - \frac{1}{2} x^2 & \mbox{if } x \in [-\frac{1}{2}, \frac{1}{4}] \\
        \frac{3}{16} + \frac{1}{2}(x - \frac{1}{2})^2 & \mbox{otherwise.}
    \end{array}
\right.
\end{align*}
$f$ is $\mathcal{C}^1$ and $1$-weakly convex and not convex with an unique minimum global and three critical points. Its Moreau envelope writes, for $\gamma \in (0, 1)$
\begin{align*}
    f^\gamma(x) &= \left\{
    \begin{array}{ll}
        \frac{(x + 1)^2}{2(1+\gamma)} & \mbox{if } x \le -\frac{1-\gamma}{2} \\
        \frac{1}{4} - \frac{x^2}{2(1 - \gamma)} & \mbox{if } x \in [-\frac{1-\gamma}{2}, \frac{1-\gamma}{4}] \\
        \frac{3}{16} + \frac{(x - \frac{1}{2})^2}{2(1+\gamma)} & \mbox{otherwise.}
    \end{array}
\right. \end{align*}
The corresponding proximal operator is
\begin{align*}
\Prox{\gamma f}{x} &= \left\{
    \begin{array}{ll}
        \frac{x - \gamma}{1+\gamma} & \mbox{if } x \le -\frac{1-\gamma}{2} \\
        \frac{x}{1 - \gamma} & \mbox{if } x \in [-\frac{1-\gamma}{2}, \frac{1-\gamma}{4}] \\
        \frac{x +\frac{\gamma}{2}}{1+\gamma} & \mbox{otherwise.}
    \end{array}
\right.
\end{align*}

On Figure~\ref{fig:example_dis}, we present the graph of the function $f$ and $f^\gamma$ for some $\gamma \in (0, 1)$. We observe that $f^\gamma$ is weakly convex but not convex. Moreover, the critical points and the global minimum are preserved.

\begin{figure}[!ht]
\centering
\includegraphics[width=\textwidth]{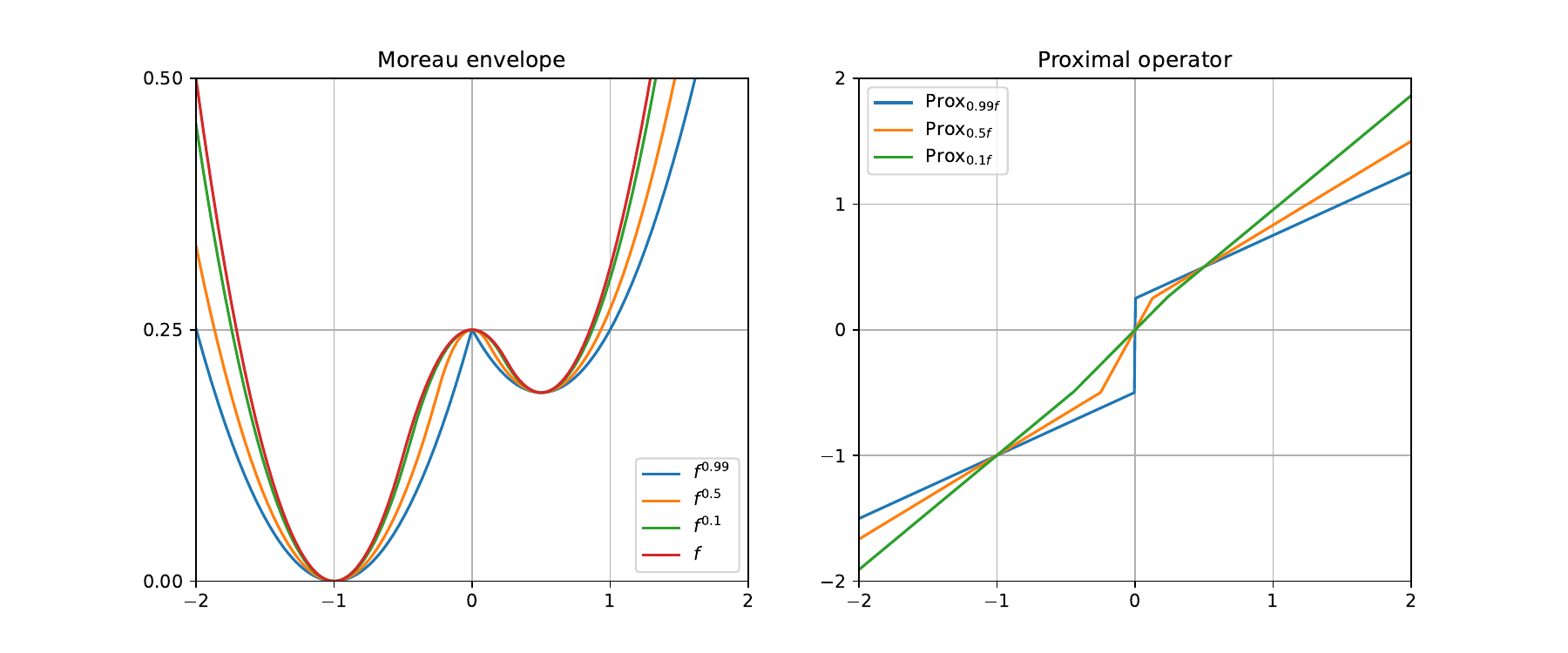}
\caption{Graph of the function $f$, $f^\gamma$ for $\gamma \in \{0.1, 0.25, 0.49\}$ and their associated proximal operators.}
\label{fig:example_dis}
\end{figure}
\end{example}

\begin{example}\label{ex:exp_non_convex}
As a third example, we study the function
\begin{align*}
    u(x) = \frac{1}{2}\left(e^{-x} - x^2\right)
\end{align*}
$u$ is $\mathcal{C}^{\infty}$ and $1$-weakly convex and not convex. After some computation, we obtain that, for $\gamma \in (0, 1)$
\begin{align*}
    u^\gamma(x) &= \frac{1-\gamma}{\gamma} \left(W\left(\frac{\gamma}{2(1-\gamma)}e^{-\frac{x}{1-\gamma}}\right) +\frac{1}{2}W^2\left(\frac{\gamma}{2(1-\gamma)}e^{-\frac{x}{1-\gamma}}\right) \right) - \frac{x^2}{2(1-\gamma)}\\
    \Prox{\gamma u}{x} &= \frac{x}{1 - \gamma} + W\left(\frac{\gamma}{2(1-\gamma)}e^{-\frac{x}{1-\gamma}} \right),
\end{align*}
where $W$ is the Lambert function defined by $W(x)e^{W(x)} = x$.

In this case, $u^{\gamma}$ is $\mathcal{C}^{\infty}$, after computations we get
\begin{align*}
    \frac{\partial u^\gamma}{\partial x}(x) &= -\frac{1}{1-\gamma}\left(x + \frac{1-\gamma}{\gamma} W\left( \frac{\gamma}{2(1-\gamma)} e^{-\frac{x}{1-\gamma}}\right) \right) \\
    \frac{\partial^2 u^\gamma}{\partial^2 x}(x) &= \frac{1}{1-\gamma}\left(\frac{2}{\gamma} \frac{W\left( \frac{\gamma}{2(1-\gamma)} e^{-\frac{x}{1-\gamma}}\right)}{1 + W\left( \frac{\gamma}{2(1-\gamma)} e^{-\frac{x}{1-\gamma}}\right)} - 1\right).
\end{align*}
We see that $\frac{\partial^2 u^\gamma}{\partial^2 x} \ge - \frac{1}{1 - \gamma}$, so that $u^\gamma$ is $\frac{1}{1 - \gamma}$-weakly convex, confirming the result of Lemma~\ref{lemma:moreau_envelop_weakly_convex}.

On Figure~\ref{fig:example_exp_non_cvx}, we present the graph of the function $u$ and $u^\gamma$ for some $\gamma \in (0, 1)$. We observe that $u^\gamma$ is weakly convex but not convex. Moreover, the critical points are preserved.

\begin{figure}[!ht]
\centering
\includegraphics[width=\textwidth]{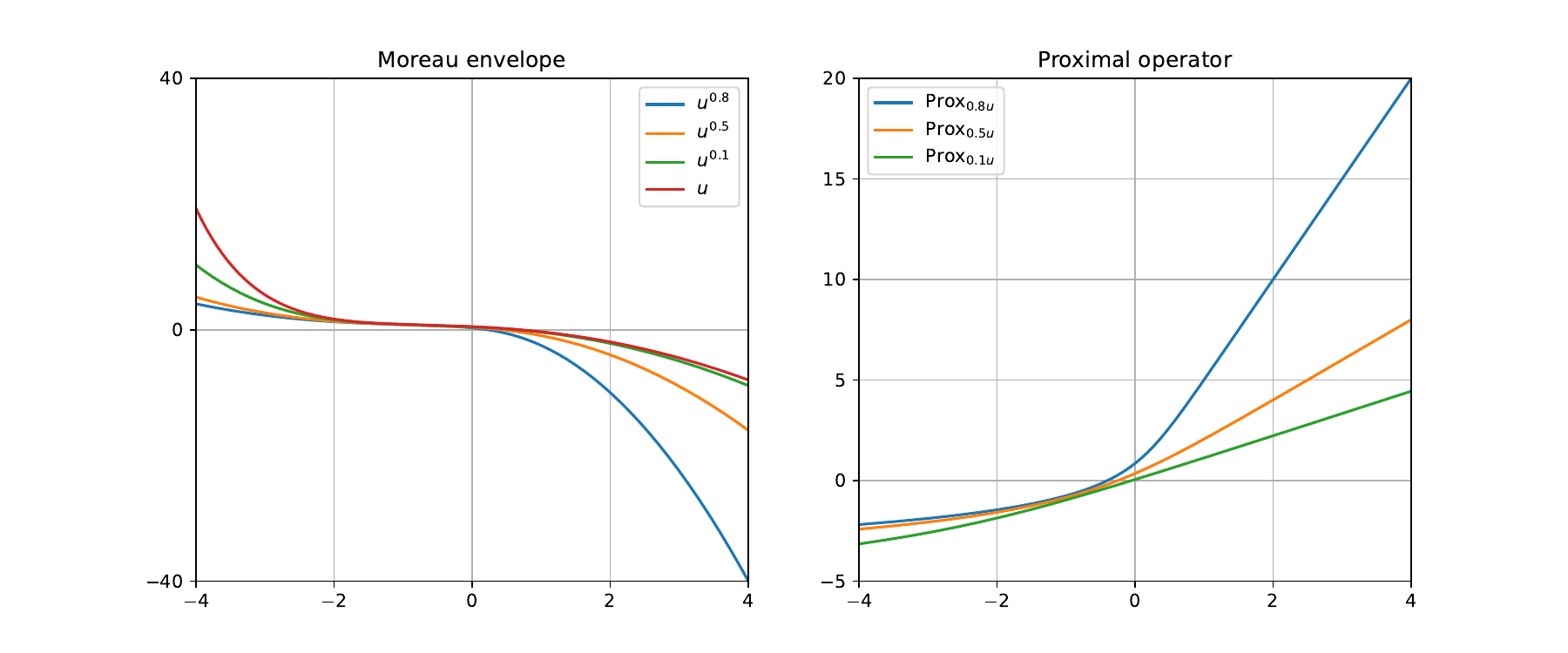}
\caption{Graph of the function $u$, $u^\gamma$ for $\gamma \in \{0.1, 0.5, 0.8\}$ and their associated proximal operators.}
\label{fig:example_exp_non_cvx}
\end{figure}
\end{example}

\begin{example}\label{ex:periodic}
We now introduce the function $\phi$ for $x \in \R$ by
\begin{align}\label{eq:def_phi}
    \phi(x) = \left| x - 2 \lfloor \frac{x+1}{2} \rfloor \right| \in [0,1]
\end{align}
and then  study the function
\begin{align*}
    v(x) = \left\{
    \begin{array}{ll}
        \frac{1}{2}\phi(x) & \mbox{if } \phi(x) \in [0,\frac{1}{2}] \\
        \frac{3}{8} - \frac{1}{2}(\phi(x) - 1)^2 & \mbox{otherwise.}
    \end{array}
    \right.
\end{align*}
The function $v$ is non differentiable on the set $\{2n | n \in \N\}$ and $1$-weakly convex and not convex.
Moreover $v$ is $2$-periodic and even. By its definition in equation~\eqref{eq:moreau_env}, it is clear that the Moreau envelope is also periodic and even. It writes, for $\gamma \in (0, 1)$
\begin{align*}
    v^\gamma(x) &= \left\{
    \begin{array}{ll}
        \frac{\phi^2(x)}{2\gamma} & \mbox{if } \phi(x) \in [0,\frac{\gamma}{2}] \\
        \frac{1}{2}\phi(x) - \frac{\gamma}{8} & \mbox{if } \phi(x) \in [\frac{\gamma}{2}, \frac{1+\gamma}{2}] \\
        \frac{3}{8} - \frac{(\phi(x) - 1)^2}{2(1 - \gamma)}  & \mbox{otherwise}
    \end{array}
    \right.\\
    \end{align*}
and its proximal operator is obtained as
    \begin{align*}
    \Prox{\gamma v}{x} &= \left\{
    \begin{array}{ll}
        0 & \mbox{if } \phi(x) \in [0,\frac{\gamma}{2}] \\
        \phi(x) - \frac{\gamma}{2} & \mbox{if } \phi(x) \in [\frac{\gamma}{2}, \frac{1+\gamma}{2}] \\
        \frac{\phi(x) - \gamma}{1 - \gamma}  & \mbox{otherwise.}
    \end{array}
    \right.
\end{align*}

On Figure~\ref{fig:example_periodic}, we present the graph of the function $v$ and $v^\gamma$ for some $\gamma \in (0, 1)$. We observe that $v^\gamma$ is periodic, pair, weakly convex but not convex. Moreover, the critical points and global minimizers are preserved.

\begin{figure}[!ht]
\centering
\includegraphics[width=\textwidth]{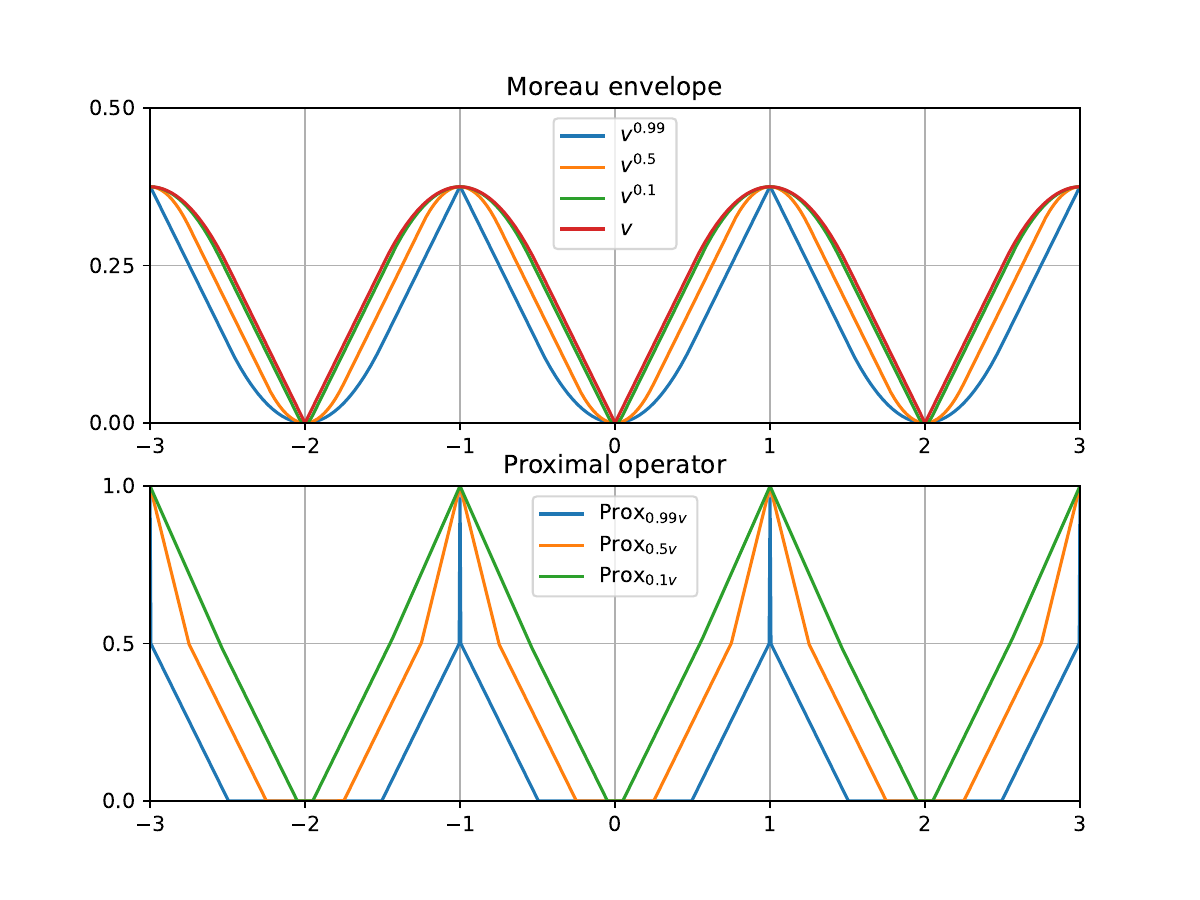}
\caption{Graph of the function $v$, $v^\gamma$ for $\gamma \in \{0.1, 0.5, 0.99\}$ and associated proximal operators.}
\label{fig:example_periodic}
\end{figure}
\end{example}

\begin{example}\label{ex:periodic_non_diff}
We now study the function $\psi = \frac{1}{2}\phi$ introduced in Equation~\eqref{eq:def_phi}
The function $\psi$ is non differentiable on the set $\{n | n \in \N\}$ and not weakly convex.
Moreover $\psi$ is $2$-periodic and even. Therefore the Moreau envelope is also $2$-periodic and even. It writes, for $\gamma \in (0, 1)$
\begin{align*}
    \psi^\gamma(x) &= \left\{
    \begin{array}{ll}
        \frac{\phi^2(x)}{2\gamma} & \mbox{if } \phi(x) \in [0,\frac{\gamma}{2}] \\
        \frac{1}{2}\phi(x) - \frac{\gamma}{8} & \mbox{otherwise}
    \end{array}
    \right.\\
    \end{align*}
and its proximal operator is obtained as
    \begin{align*}
    \Prox{\gamma \psi}{x} &= \left\{
    \begin{array}{ll}
        0 & \mbox{if } \phi(x) \in [0,\frac{\gamma}{2}] \\
        \phi(x) - \frac{\gamma}{2} & \mbox{otherwise.}
    \end{array}
    \right.
\end{align*}

On Figure~\ref{fig:example_periodic_non_diff}, we present the graph of the function $\psi$ and $\psi^\gamma$ for some $\gamma \in (0, 1)$. We observe that $\psi^\gamma$ is periodic, pair but not differentiable. Moreover, the critical points and global minimizers are preserved.

\begin{figure}[!ht]
\centering
\includegraphics[width=\textwidth]{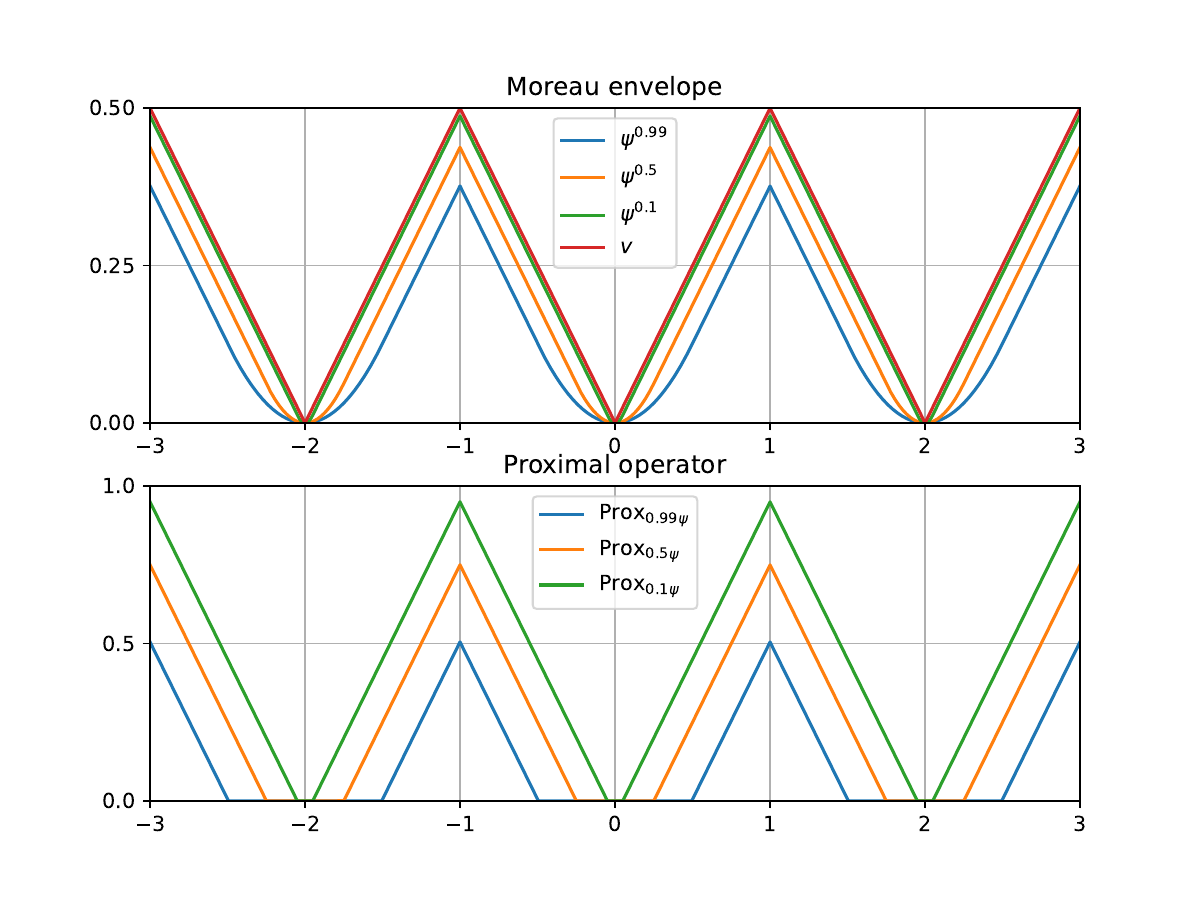}
\caption{Graph of the function $\psi$, $\psi^\gamma$ for $\gamma \in \{0.1, 0.5, 0.99\}$ and associated proximal operators.}
\label{fig:example_periodic_non_diff}
\end{figure}
\end{example}

\begin{example}
In many inverse problem~\cite{chen2001atomic,candes2005decoding}, the restoration task is reformulated as an optimization problem of the form
\begin{align}\label{eq:inv_pb}
    \inf_{x \in \R^d} \|\mathcal{A}(x) - y\|_1 + \phi(x),
\end{align}
with $\phi : \R^d \to \R$ a smooth regularization function, $\mathcal{A}: \R^d \to \R^m$ a degradation operator and $y \in \R^m$ an observation.

However, the previous objective function is not smooth. Therefore, the authors of~\cite{slavakis2011robust,chen2014moreau} propose to replace the convex $g = \|\cdot\|_1$ function by its Moreau envelope $g^\gamma = \frac{1}{\gamma}L_{\gamma}$, with $L_\gamma$ the Hubert loss~\cite{huber1992robust} defined by
\begin{align*}
    L_{\gamma}(x) &= \left\{
    \begin{array}{ll}
        \frac{1}{2}\|x\|_2^2 & \mbox{if } |x|<\gamma\\
        \gamma (\|x\|_1 - \frac{d\gamma}{2}) & \mbox{otherwise,}
    \end{array}
    \right.
\end{align*}
where $d$ is the dimension of the space.
It allows to run gradient descent algorithms to approximate solution of problem~\eqref{eq:inv_pb}. 
The Moreau envelop has also been proposed in the same context for smoothing convex constraints~\cite{yu2017moreau}
\end{example}

\begin{example}
A well-established line of research~\cite{roberts1996exponential} use Langevin algorithm to sample a target distribution defined by
\begin{align*}
    \pi \propto e^{-f-g},
\end{align*}
with $f$ a data-fidelity term and $g$ a weakly-convex regularization. 
Such an algorithm can be written as
\begin{align*}
    x_{k+1} = x_k - \delta \nabla f(x_k) - \delta \nabla g(x_k) + \sqrt{2\delta}z_{k+1},
\end{align*}
with $z_{k+1} \sim \mathcal{N}(0, I_d)$.
However, in the case of non-differentiable regularization, this algorithm can not be run. Therefore, $g$ is approximated by $g^\gamma$~\cite{pereyra2016proximal,durmus2018efficient,luu2021sampling, renaud2025stability,crucinio2025optimal,habring2025diffusion}, and thanks to Proposition~\ref{prop:nabla_g_weakly_cvx_appendix}, the previous algorithm can be written as
\begin{align*}
    x_{k+1} = x_k - \delta \nabla f(x_k) - \frac{\delta}{\gamma}\left( x_k - \Prox{\gamma g}{x_k} \right) + \sqrt{2\delta}z_{k+1}.
\end{align*}
This type of algorithm has recently shown state-of-the art performances for sampling image posterior distributions~\cite{renaud2025stability}.
\end{example}

\section{Conclusion}
In this document, we have presented the main properties of the Moreau envelope for weakly convex function. We show that the Moreau envelope is regular in the sense that it is differentiable with respect to $\gamma \in (0, \frac{1}{\rho})$ and  $x \in \R^d$, and that its derivatives can be explicitly computed. We highlight the intrinsic connections between the Moreau envelope and the convex conjugate. Moreover, we show that the Moreau envelope preserves key objects of interest in optimization, namely the minimizers and the critical points. Moreover, we study the second order properties and demonstrate that the image of the proximal operator is almost convex.

\section*{Acknowledgements}
This study has been carried out with financial support from the French Direction G\'en\'erale de l’Armement.
We thanks Olivier Ley and Antonin Chambolle for the useful discussions and pointing us some references.

\bibliography{ref}
\bibliographystyle{abbrv}

\end{document}